\documentclass[10pt,reqno,oneside]{amsart}
\usepackage{graphicx,verbatim}
\usepackage{amssymb,cite}
\usepackage{epstopdf}
\usepackage{graphicx}
\usepackage{amsmath}
\usepackage{amsthm, enumerate}
\usepackage{mathrsfs}
\usepackage{caption}
\usepackage{subcaption}
\usepackage[normalem]{ulem}   
\allowdisplaybreaks

\chardef\forshowkeys=0
\chardef\showllabel=0
\chardef\refcheck=0
\chardef\sketches=0
\chardef\showfont=1         

\usepackage{float}

\floatstyle{boxed}

\usepackage{marginnote}

\usepackage[colorlinks=true, pdfstartview=FitV, linkcolor=blue, citecolor=blue, urlcolor=blue]{hyperref}

\ifnum\forshowkeys=1
  
  \usepackage[notref,notcite,color]{showkeys}
\fi

\author[M.S.~Ayd{\i}n]{Mustafa Sencer Ayd{\i}n}
\address{Department of Mathematics, University of Southern California, Los Angeles, CA 90089}
\email{maydin@usc.edu}

\title[Inviscid limit on $L^p$-based Sobolev conormal spaces]{Inviscid limit on $L^p$-based Sobolev conormal spaces for the 3D Navier-Stokes equations with the Navier boundary conditions}

\usepackage{enumitem}
\usepackage{datetime}
\usepackage{fancyhdr}
\usepackage{comment}
\allowdisplaybreaks
\usepackage[margin=1in]{geometry}
\usepackage{amsmath, amsthm, amssymb}
\usepackage{times}
\usepackage{graphicx}
\usepackage[usenames,dvipsnames,svgnames,table]{xcolor}

\ifnum\showfont=0
  \usepackage{fontspec}
\fi

\ifnum\refcheck=1
\usepackage{refcheck}
\fi

\begin{document}
 
\def\inprogress{{\colg IN PROGRESS} }
\def\bnew{\colr {}}
\def\enew{\colb {}}
\def\bold{\colu {}}
\def\eold{\colb{}}
\def\YY{X}
\def\OO{\mathcal O}
\def\SS{\mathbb S}
\def\CC{\mathbb C}
\def\RR{\mathbb R}
\def\TT{\mathbb T}
\def\ZZ{\mathbb Z}
\def\HH{\mathbb H}
\def\RSZ{\mathcal R}
\def\LL{\mathcal L}
\def\SL{\LL^1}
\def\ZL{\LL^\infty}
\def\GG{\mathcal G}
\def\tt{\langle t\rangle}
\def\erf{\mathrm{Erf}}
\def\mgt#1{\textcolor{magenta}{#1}}
\def\ff{\rho}
\def\gg{G}
\def\sqrtnu{\sqrt{\nu}}
\def\ww{w}
\def\ft#1{#1_\xi}
\def\ges{\gtrsim}
\renewcommand*{\Re}{\ensuremath{\mathrm{{\mathbb R}e\,}}}
\renewcommand*{\Im}{\ensuremath{\mathrm{{\mathbb I}m\,}}}

\ifnum\showllabel=1
\def\llabel#1{\marginnote{\color{gray}\rm(#1)}[-0.0cm]\notag}
\else
\def\llabel#1{\notag}
\fi

\newcommand{\norm}[1]{\left\|#1\right\|}
\newcommand{\nnorm}[1]{\lVert #1\rVert}
\newcommand{\abs}[1]{\left|#1\right|}
\newcommand{\NORM}[1]{|\!|\!| #1|\!|\!|}
\theoremstyle{plain}
\newtheorem{theorem}{Theorem}[section]
\newtheorem{Theorem}{Theorem}[section]
\newtheorem{corollary}[theorem]{Corollary}
\newtheorem{Corollary}[theorem]{Corollary}
\newtheorem{proposition}[theorem]{Proposition}
\newtheorem{Proposition}[theorem]{Proposition}
\newtheorem{Lemma}[theorem]{Lemma}
\newtheorem{lemma}[theorem]{Lemma}
\theoremstyle{definition}
\newtheorem{definition}{Definition}[section]
\newtheorem{remark}[Theorem]{Remark}
\def\theequation{\thesection.\arabic{equation}}
\numberwithin{equation}{section}
\definecolor{mygray}{rgb}{.6,.6,.6}
\definecolor{myblue}{rgb}{9, 0, 1}
\definecolor{colorforkeys}{rgb}{1.0,0.0,0.0}
\newlength\mytemplen
\newsavebox\mytempbox
\makeatletter
\newcommand\mybluebox{%
  \@ifnextchar[
  {\@mybluebox}%
  {\@mybluebox[0pt]}}
\def\@mybluebox[#1]{%
  \@ifnextchar[
  {\@@mybluebox[#1]}%
  {\@@mybluebox[#1][0pt]}}
\def\@@mybluebox[#1][#2]#3{
  \sbox\mytempbox{#3}%
  \mytemplen\ht\mytempbox
  \advance\mytemplen #1\relax
  \ht\mytempbox\mytemplen
  \mytemplen\dp\mytempbox
  \advance\mytemplen #2\relax
  \dp\mytempbox\mytemplen
  \colorbox{myblue}{\hspace{1em}\usebox{\mytempbox}\hspace{1em}}}
 \makeatother
\def\XX{{\mathcal X}}
\def\XXT{{\mathcal X}_T}
\def\XXTzero{{\mathcal X}_{T_0}}
\def\XXi{{\mathcal X}_\infty}
\def\YY{{\mathcal Y}}
\def\YYT{{\mathcal Y}_T}
\def\YYTzero{{\mathcal Y}_{T_0}}
\def\YYi{{\mathcal Y}_\infty}
\def\cc{\text{c}}
\def\rr{r}
\def\weaks{\text{\,\,\,\,\,\,weakly-* in }}
\def\inn{\text{\,\,\,\,\,\,in }}
\def\cof{\mathop{\rm cof\,}\nolimits}
\def\Dn{\frac{\partial}{\partial N}}
\def\Dnn#1{\frac{\partial #1}{\partial N}}
\def\tdb{\tilde{b}}
\def\tda{b}
\def\qqq{u}
\def\lat{\Delta_2}
\def\biglinem{\vskip0.5truecm\par==========================\par\vskip0.5truecm}
\def\inon#1{\hbox{\ \ \ \ \ \ \ }\hbox{#1}}                
\def\onon#1{\inon{on~$#1$}}
\def\inin#1{\inon{in~$#1$}}
\def\FF{F}
\def\andand{\text{\indeq and\indeq}}
\def\ww{w(y)}
\def\ll{{\color{red}\ell}}
\def\ee{\epsilon_0}
\def\startnewsection#1#2{ \section{#1}\label{#2}\setcounter{equation}{0}}   
\def\loc{\text{loc}}
\def\nnewpage{ }
\def\sgn{\mathop{\rm sgn\,}\nolimits}    
\def\Tr{\mathop{\rm Tr}\nolimits}    
\def\div{\mathop{\rm div}\nolimits}
\def\curl{\mathop{\rm curl}\nolimits}
\def\dist{\mathop{\rm dist}\nolimits}  
\def\supp{\mathop{\rm supp}\nolimits}
\def\indeq{\quad{}}           
\def\period{.}                       
\def\semicolon{\,;}                  
\def\colr{\color{red}}
\def\colrr{\color{black}}
\def\colb{\color{black}}
\def\coly{\color{lightgray}}
\definecolor{colorgggg}{rgb}{0.1,0.5,0.3}
\definecolor{colorllll}{rgb}{0.0,0.7,0.0}
\definecolor{colorhhhh}{rgb}{0.3,0.75,0.4}
\definecolor{colorpppp}{rgb}{0.7,0.0,0.2}
\definecolor{coloroooo}{rgb}{0.45,0.0,0.0}
\definecolor{colorqqqq}{rgb}{0.1,0.7,0}
\def\colg{\color{colorgggg}}
\def\collg{\color{colorllll}}
\def\cole{\color{colorpppp}}
\def\cole{\color{coloroooo}}
\def\coleo{\color{colorpppp}}
\def\cole{\color{black}}
\def\colu{\color{blue}}
\def\colc{\color{colorhhhh}}
\def\colW{\colb}   
\definecolor{coloraaaa}{rgb}{0.6,0.6,0.6}
\def\colw{\color{coloraaaa}}
\def\comma{ {\rm ,\qquad{}} }            
\def\commaone{ {\rm ,\quad{}} }          
\def\lec{\lesssim}
\def\nts#1{{\color{red}\hbox{\bf ~#1~}}} 
\def\ntsf#1{\footnote{\color{colorgggg}\hbox{#1}}}
\def\ntsik#1{{\color{purple}\hbox{\bf ~#1~}}} 
\def\blackdot{{\color{red}{\hskip-.0truecm\rule[-1mm]{4mm}{4mm}\hskip.2truecm}}\hskip-.3truecm}
\def\bluedot{{\color{blue}{\hskip-.0truecm\rule[-1mm]{4mm}{4mm}\hskip.2truecm}}\hskip-.3truecm}
\def\purpledot{{\color{colorpppp}{\hskip-.0truecm\rule[-1mm]{4mm}{4mm}\hskip.2truecm}}\hskip-.3truecm}
\def\greendot{{\color{colorgggg}{\hskip-.0truecm\rule[-1mm]{4mm}{4mm}\hskip.2truecm}}\hskip-.3truecm}
\def\cyandot{{\color{cyan}{\hskip-.0truecm\rule[-1mm]{4mm}{4mm}\hskip.2truecm}}\hskip-.3truecm}
\def\reddot{{\color{red}{\hskip-.0truecm\rule[-1mm]{4mm}{4mm}\hskip.2truecm}}\hskip-.3truecm}
\def\gdot{\greendot}
\def\tdot{\gdot}
\def\bdot{\bluedot}
\def\ydot{\cyandot}
\def\rdot{\reddot}
\def\fractext#1#2{{#1}/{#2}}
\def\ii{\hat\imath}
\def\fei#1{\textcolor{blue}{#1}}
\def\vlad#1{\textcolor{cyan}{#1}}
\def\igor#1{\text{{\textcolor{colorqqqq}{#1}}}}
\def\igorf#1{\footnote{\text{{\textcolor{colorqqqq}{#1}}}}}
\newcommand{\p}{\partial}
\newcommand{\UE}{U^{\rm E}}
\newcommand{\PE}{P^{\rm E}}
\newcommand{\KP}{K_{\rm P}}
\newcommand{\uNS}{u^{\rm NS}}
\newcommand{\vNS}{v^{\rm NS}}
\newcommand{\pNS}{p^{\rm NS}}
\newcommand{\omegaNS}{\omega^{\rm NS}}
\newcommand{\uE}{u^{\rm E}}
\newcommand{\vE}{v^{\rm E}}
\newcommand{\pE}{p^{\rm E}}
\newcommand{\omegaE}{\omega^{\rm E}}
\newcommand{\ua}{u_{\rm   a}}
\newcommand{\va}{v_{\rm   a}}
\newcommand{\omegaa}{\omega_{\rm   a}}
\newcommand{\ue}{u_{\rm   e}}
\newcommand{\ve}{v_{\rm   e}}
\newcommand{\omegae}{\omega_{\rm e}}
\newcommand{\omegaeic}{\omega_{{\rm e}0}}
\newcommand{\ueic}{u_{{\rm   e}0}}
\newcommand{\veic}{v_{{\rm   e}0}}
\newcommand{\up}{u^{\rm P}}
\newcommand{\vp}{v^{\rm P}}
\newcommand{\tup}{{\tilde u}^{\rm P}}
\newcommand{\bvp}{{\bar v}^{\rm P}}
\newcommand{\omegap}{\omega^{\rm P}}
\newcommand{\tomegap}{\tilde \omega^{\rm P}}
\renewcommand{\up}{u^{\rm P}}
\renewcommand{\vp}{v^{\rm P}}
\renewcommand{\omegap}{\Omega^{\rm P}}
\renewcommand{\tomegap}{\omega^{\rm P}}
\def\hh{\text{h}}
\def\cco{\text{co}}

\begin{abstract}
We establish uniform bounds and the inviscid limit in $L^p$-based Sobolev conormal spaces for the solutions of the Navier-Stokes equations with the Navier boundary conditions in the half-space. We extend the vanishing viscosity results of~\cite{BdVC1} and~\cite{AK1} by weakening the normal and the conormal regularity assumptions, respectively.
We require the initial data to be Lipschitz with three integrable conormal derivatives. We also assume that the initial normal derivative has one or two integrable conormal derivative depending on the sign of the friction coefficient.
Finally, we establish the existence and uniqueness of the Euler equations with a bounded normal derivate, two bounded conormal derivatives, and three integrable conormal derivatives.
\end{abstract}

\maketitle

\startnewsection{Introduction}{sec.int}
 
 The inviscid limit is the 
 study of the limiting behavior of the solutions to the incompressible Navier-Stokes equations
 \begin{align}
  \partial_t u^\nu - \nu \Delta u^\nu + u^\nu \cdot \nabla u^\nu + \nabla p^\nu =0 
  \comma
  \nabla \cdot u^\nu = 0 \comma (x,t) \in \Omega \times (0,T)
  ,
  \label{NSE}
 \end{align}
 where $\Omega \subseteq \mathbb{R}^3$ as $\nu \to 0$.
 Assuming that $\{u^\nu\}_{\nu}$ is convergent, one may
 expect to recover a solution $u$ for the incompressible Euler equations
 \begin{align}
 	u_t + u\cdot \nabla u + \nabla p =0
 	\comma \nabla \cdot u = 0
 	\inin{\Omega}
 	.
 	\label{euler}
 \end{align}
 This problem has been open to various extents depending on
 the geometry and the dimension of the physical domain $\Omega$, the regularity imposed on the solutions,
 and the notion of convergence.
 When $\partial \Omega \neq 0$,
 the Euler equations \eqref{euler} are coupled with the slip boundary condition
 \begin{align}
 	u \cdot n = 0 
 	\onon{\partial\Omega}
 	\label{eulerb}
 \end{align}
 that models the tangential movement of the fluid particles along the boundary.
 For the Navier-Stokes equations, one possibility is to impose the no-slip boundary condition
 \begin{align}
 	u^\nu = 0, \onon{\partial \Omega},
 	\label{noslip}
 \end{align}
 which prevents the movement of the fluid particles at the boundary.
 Therefore, one may expect a formation of boundary layers or accumulation 
 of vorticity for regimes with low viscosity (or high Reynolds number).
 In fact, Kelliher proved in~\cite{Ke} that the inviscid limit (in the energy norm)
 is equivalent to the formation of a vortex sheet on the boundary. 
 Although being open to a large extent, the limiting 
 behavior of $u^\nu$ can be mathematically described under certain functional settings 
 when \eqref{noslip} is imposed; see, for example,~\cite{CLNV, DN, K3, KVW, M, SC1, SC2, TW}
 and the references in~\cite{MM}.
 
 One may also couple \eqref{NSE} with the Navier-boundary conditions
 \begin{align}
 	u^\nu\cdot n = 0 \comma 
 	\left(\frac{1}{2}(\nabla u^\nu + \nabla^T u^\nu)\cdot n\right)_\tau
 	=-\mu u^\nu_\tau
 	\onon{\partial\Omega}
 	,
 	\label{navierbdry}
 \end{align}
 where $\mu \in \mathbb{R}$ is a constant, $n$ is the outward unit normal vector, and 
 $v_\tau = v - (v \cdot n)n$
 is the tangential part of~$v$.
 Allowing tangential movement, the Navier-boundary conditions
 assert that the viscous stress is proportional to the tangential velocity by 
 a friction coefficient. Mathematically, this condition
 implies that the normal derivative can be written as a sum of tangential derivatives 
 and lower order terms on the boundary. Therefore, one may expect the formation
 of a so-called ``weak'' boundary layer. We refer the reader to~\cite{GK, IS}
 for the mathematical description of such boundary layers and the corresponding
 inviscid limit results. 
 
 The vanishing viscosity limit in the energy norm was first
 studied Iftimie and Planas in~\cite{IP}, where the authors establish
 the convergence of strong Navier-Stokes solutions
 in the energy norm to the strong Euler solution. 
 Later, this problem has been studied in~\cite{BS1, BS2, CQ, NP, WXZ, X, XX1, XX2}.
 We remark two common grounds of these results. First is that their functional
 setting is $L^2$-based, and second is that their assumptions on the initial data
 yield a strong solution for both \eqref{NSE} and \eqref{euler}. 
 
 Regarding the Lebesgue exponent, the authors in~\cite{BdVC1} 
 are the first to consider an $L^p$ with $p\neq 2$ functional setting
 in three dimensions with the Navier boundary condition.
 On the periodic channel, they assumed that the initial datum belongs
 to $W^{3,p}$, for $p>3/2$, and established that
 \begin{align}
 	u^\nu \to u \text{ in } C([0,T];W^{s,p}(\Omega))
 	\text{ and } u^\nu \rightharpoonup u \text{ in } L^\infty(0,T, W^{3,p}(\Omega)) \text{ weak* },
 	\llabel{bdvc} 
 \end{align}  
 for $s<3$ and $T>0$.
 Later, in~\cite{BdVC2}, the authors extended this result to integer derivatives $k\ge 3$
 upon considering $p \ge 2$. These results need not hold in curved domains; 
 see~\cite{BdVC3, BdVC4}
 for the negative results on the three-dimensional sphere. 
 
 Regarding the Euler solutions, Masmoudi and Rousset in~\cite{MR1}
 presented a functional framework allowing only a
 single normal derivative.  
 In particular, for initial data satisfying
 \begin{align}
 	(u,\nabla u)|_{t=0}  \in H^7_\cco(\Omega) \times (H^6_\cco(\Omega) \cap W^{1,\infty}_\cco(\Omega))
 	,\llabel{mrassumption}
 \end{align}
 (see the definitions Sobolev conormal spaces in~\eqref{con.def} and~\eqref{norm} below),
 they established the uniform convergence of $u^\nu$ to $u$,
 which also shows that the Euler equations are well-posed for
 this class of initial data.
 Later, in~\cite{AK1}, we weakened the assumptions 
 required on the initial data to establish uniform bounds on solutions and
 the uniform convergence in the vanishing viscosity limit. Namely,
 we assumed  
 \begin{align}
 	\begin{split}
	(u,\nabla u)|_{t=0}  &\in H^5_\cco(\Omega)  \times (H^2_\cco(\Omega) \cap L^\infty(\Omega))
	\comma \hspace{2.1cm} \mu \in \mathbb{R},
	\\
	(u,\nabla u)|_{t=0}  &\in (H^4_\cco(\Omega) \cap W^{2,\infty}_{\cco}(\Omega)) \times (H^1_\cco(\Omega) \cap L^\infty(\Omega))
	\comma \mu\ge 0,
	\llabel{0ourassumption}
	\end{split}
\end{align}
and established the inviscid limit in $L^\infty((0,T)\times \Omega)\cap L^\infty(0,T;L^2(\Omega))$.
We discuss in Section~\ref{sec.pre} the difference between the regularity required on the initial data
depending on the sign of $\mu$.

Our aim in the current work is to unify the $L^p$
and Sobolev conormal settings initiated by~\cite{BdVC1} and~\cite{MR1}, respectively,
for the study of the inviscid limit under the Navier-boudary conditions.
In particular, for $\delta>0$, we assume that
 \begin{align}
	\begin{split}
		(u,\nabla u)|_{t=0}  &\in 
		 (L^2(\Omega) \cap W_\cco^{3,3+\delta}(\Omega) \cap W_\cco^{3,6+2\delta}(\Omega))  \times (W^{2,3+\delta}_\cco(\Omega) \cap L^\infty(\Omega))
		\comma \mu \in \mathbb{R},
		\\
		(u,\nabla u)|_{t=0}  &\in 
		 (L^2(\Omega) \cap W_\cco^{6,3+\delta}(\Omega) \cap W^{2,\infty}_{\cco}(\Omega)) \times (W^{1,6}_\cco(\Omega) \cap L^\infty(\Omega))
		\comma \hspace{0.7cm} \mu\ge 0,
		\label{ourassumption}
	\end{split}
\end{align}
and establish the inviscid limit in $L^\infty(0,T;L^p(\mathbb{R}^3_+))$,
for $p\in [2,\infty]$; see Theorems~\ref{T01} and~\ref{T03}.
To the best of our knowledge, our work is the first to establish
the well-posedness of the 3D Navier-Stokes equations in $L^p$-based Sobolev conormal spaces. In addition, 
our vanishing viscosity result extends the one 
presented in~\cite{BdVC1} since we 
require only boundedness of one normal derivative. Next, we also improve upon the results in~\cite{AK1} and~\cite{MR1}
by reducing the conormal differentiability requirements on the initial data. 

A large part of this work is devoted to 
obtaining uniform-in-$\nu$ estimates when propogating \eqref{ourassumption}.
The uniform bounds established in~\cite{BdVC1} in the $L^p$-based Sobolev setting, 
rely on the symmetries
of the Navier-Stokes equation.
Such symmetries are unavailable to us  
since our functional setting is anisotropic 
and the commutator terms resulting from the conormal derivatives 
are not zero. 
In the $L^2$-based setting, 
we utilized energy methods to establish uniform bounds; see~\cite[Proposition~2.2]{AK1}.
The main challenge was the analysis of the commutator terms resulting
from the diffusive or the advective part of the equation. 
In the current work, we rely on the energy method
to estimate the conormal derivatives of $u$, $\nabla u$, and $\nabla p$.
However, unlike in the $L^2$ case, energy estimates on $Z^\alpha u$
do not yield a control over the term $\nu \Vert \nabla Z^\alpha u\Vert_{L^p}^p$,
where $Z^\alpha$ is a conormal derivative of order $|\alpha|$.
Indeed, in Proposition~\ref{P.Con}, we establish an estimate of the form
\begin{align}
	\begin{split}
		\Vert  Z^\alpha u(t)&\Vert_{L^p}^{p}
		+c_0\nu \sum_{0\le |\alpha| \le 3}\left(\int_0^t \int_\Omega \left(|\nabla Z^\alpha u|^2 |Z^\alpha u|^{p-2}
		+ |\nabla |Z^\alpha u|^\frac{p}{2}|^2\right) \,dxds\right)
		\\&\lec
		\Vert Z^\alpha u_0\Vert_{L^p}^p
		+\int_0^t I(s) \,ds
		+\nu^{\frac{p}{2}}\int_0^t \Vert \partial_z Z^{\beta} u(s)\Vert_{L^p}^p\,ds
		,
		\label{0EQ.Con}
	\end{split}
\end{align}
where $Z^\beta$ is a conormal derivative of order $|\alpha|-1$ and 
$I$ consists of the terms to be propogated by the Gronwall inequality.
Upon letting $p=2$, it is possible to control the normal derivative
term on the right-hand side by an induction argument; see~\cite[Proposition~3.1]{AK1}.
However, when $p>2$, this is not possible. 
Therefore, we cannot utilize the dissipative nature of the Laplacian term 
using energy methods. 

In \eqref{0EQ.Con}, the term involving $\partial_z Z^\beta u$
is multiplied by a power of $\nu$. Motivated by this, we rely on 
the maximal regularity properties of the heat equation to estimate this 
and many other higher order terms that result from energy estimates. 
However, the maximal regularity estimates introduce new challenges. 
Broadly speaking, one requires initial data in $W^{1-\frac{2}{p},p}$ 
to gain a derivative in $L^p(\Omega \times (0,T))$. Therefore, we need to quantify by $\nu$
the approximation of the initial datum; see Section~\ref{sec.main}.
The second challenge arises from the unfavorable sign of the friction coefficient $\mu$.
When $\mu$ is non-negative, \eqref{navierbdry} becomes a homogenous Robin boundary condition.
Hence, we do not run into boundary terms when we utilize maximum regularity estimates.
However, this is not the case when $\mu <0$ because estimates result in a boundary term involving
mixed fractional space-time derivates. We handle this term by establishing a trace-type inequality; see Lemma~\ref{Lt2}.
Lastly, since $\partial \Omega \neq 0$, the maximal regularity estimates require a compatible initial datum. 
In~2D, it is possible to approximate a non-compatible initial data with a sequence
of smooth compatible ones, see~\cite{CMR, LNP}, but it is unclear whether this is possible in 3D. 
Nevertheless, the inviscid limit problem with no-slip conditions and non-compatible initial data
has been considered in~\cite{ACS, GKLMN}. 

Theorems~\ref{T01} and~\ref{T03} have an implication in
addition to the inviscid limit. Namely, they also
establish the well-posedness of the Euler equations 
for the class of initial data given by \eqref{ourassumption}.
In Theorem~\ref{T04}, we further weaken the assumptions \eqref{ourassumption}
and prove that the Euler equations has a unique solution. 
Before discussing Theorem~\ref{T04}, we briefly summarize 
the previous works on the well-posedness of the
strong solutions for the three-dimensional
Euler equations. The literature dates back to
the work of Lichtenstein in~\cite{L}, where he considered $u_0$ in $C^{k,\alpha}$.
Regarding the Sobolev spaces, in~\cite{K1} and~\cite{K2}, Kato 
studied this problem for $u_0 \in H^m(\mathbb{R}^3)$, for $m\ge 3$, and $H^s(\mathbb{R}^3)$, for $s > \frac{5}{2}$, respectively.
Then, Kato and Ponce in~\cite{KP} extended this result to the $L^p$-based Sobolev spaces
$W^{s,p}(\mathbb{R}^d)$, for $s > \frac{d}{p}+1$. These results have counterparts in 
bounded domains such as~\cite{BB, KL, Te}. In addition,
well-posedness of the Euler equations has been studied in other functional settings
such as Besov or Triebel-Lizorkin spaces. We refer the reader to~\cite {C1,C2,C3,CW,GL,GLY,PP}
and the references therein. These results consider isotropic spaces
and require integrability of at least two derivatives in the normal direction,
whereas the inviscid limit type constructions in~\cite{AK1, MR1} do not.
Apart from the inviscid limit approach, the authors in~\cite{BILN} established the well-posedness of the Euler
equations in the conormal spaces. Namely,
they considered the class of initial data satisfying
\begin{align}
	(u,\nabla u,\curl u)|_{t=0}
	\in
	H^4_\cco(\Omega)
	\times
	H^3_\cco(\Omega)
	\times
	W^{1,\infty}_\cco(\Omega)
	,
	\llabel{EQ0006}
\end{align}
for general domains.
Later, in~\cite{AK2}, we extended this result 
by assuming
\begin{align}
	(u,\nabla u)|_{t=0}
	\in
	H^4_\cco(\Omega) \cap W^{2,\infty}_\cco(\Omega)
	\times
	L^\infty(\Omega),
	\llabel{EQ0007}
\end{align}
in the half-space.
The improvement in~\cite{AK2} is that
we did not require any bounded or integrable conormal derivatives on $\nabla u$. 
In our current work, we further decrease the conormal
differentiability requirement on $u_0$ and prove the existence and
uniqueness for the Euler equations in the half-space with initial data in
\begin{align}
	(u,\nabla u)|_{t=0}
	\in
	(W^{3,3+\delta}_\cco(\Omega) \cap W^{2,\infty}_\cco(\Omega)
	)\times
	L^\infty(\Omega),
	\llabel{EQ0008}
\end{align}
for $\delta>0$; see Theorem~\ref{T04}.
The construction here is follows by the a~priori
estimates and approximation by smooth solutions.
 
 \startnewsection{Preliminaries and the Main Results}{sec.pre}
 
Let $\Omega =\mathbb{R}^3_+$, and 
 denote  
 $x = (x_\hh,z) = (x_1,x_2,z) \in \Omega= \mathbb{R}^2 \times \mathbb{R}_+$.
Rewriting \eqref{NSE} and \eqref{navierbdry} for the half-space, we have
 \begin{align}
  \partial_t u^\nu - \nu \Delta u^\nu + u^\nu \cdot \nabla u^\nu + \nabla p^\nu =0 
  \comma
  \nabla \cdot u^\nu = 0  
  ,
  \label{NSE0}
 \end{align}
 for $(x,t) \in \Omega \times (0,T)$, and
 \begin{align}
  u^\nu_3 = 0 \comma \partial_z u^\nu_\hh = 2\mu u^\nu_\hh, 
  \label{hnavierbdry}
 \end{align}
 for $(x,t)\in \{z=0\}\times (0,T)$
where $u_\hh = (u_1,u_2)$.
Now, we denote $\varphi(z) = z/(1+z)$ and
 write
 $Z_1=\partial_1$, $Z_2=\partial_2$, and $Z_3=\varphi \partial_z$.
 Next, we define 
 the Sobolev conormal spaces 
 \begin{align}
 	W^{m,p}_\cco=
   W^{m,p}_\cco(\Omega)
   =
   \{f \in L^p(\Omega) : Z^\alpha f \in L^p(\Omega), \alpha \in \mathbb{N}_0^3, 0\le |\alpha|\le m  \}
   ,
   \label{con.def}
   \end{align} 
 for $p \in [1,\infty]$. We note that these are Banach spaces 
 when equipped with the norms defined by
 \begin{align}
  \begin{split}
   \Vert f\Vert_{W^{m,p}_\cco(\Omega)}^2
   =&
   \Vert f\Vert_{m,p}^p
   =
   \sum_{|\alpha|\le m} \Vert Z^\alpha f\Vert_{L^p(\Omega)}^p
   \comma 1\le p<\infty\\
   \Vert f\Vert_{W^{m,\infty}_\cco(\Omega)}
   =&
   \Vert f\Vert_{m,\infty}
   =
   \sum_{|\alpha|\le m} \Vert Z^\alpha f\Vert_{L^\infty(\Omega)}
   .
   \label{norm}
  \end{split}
 \end{align}
 We denote by $\Vert f\Vert_{L^2}$ the $L^2$ norm and by $\Vert f\Vert_{L^\infty}$
 the $L^\infty$ norm of $f$. Next, we fix a sufficienty small constant 
 $\bar{\nu}>0$ and an arbitrary $\delta>0$. Moreover, for the rest of this work,
 we assume that $T\le 1$. Now, we state our first main result. 
 
 \cole
 \begin{Theorem}[existence, uniqueness, and inviscid limit]
 	\label{T01}
 	Assume that $\nu \in (0,\bar{\nu}]$ and $\mu \in \mathbb{R}$.
 	Let $u_0 \in L^2(\Omega)\cap W^{3,3+\delta}_\cco(\Omega) \cap W^{3,6+2\delta}_\cco(\Omega)$,
 	with $\nabla u_0 \in W^{2,3+\delta}_\cco(\Omega) \times L^\infty(\Omega)$ be
 	such that $\div u_0 = 0$ on $\Omega$ and $u_0 \cdot n= 0$ and
 	$\partial_z (u_0)_\hh = 2\mu (u_0)_\hh$ on~$\partial \Omega$.
 	Then, there exists a sequence of smooth divergence-free initial data $u_0^\nu$
 	such that $u_0^\nu \to u_0$ in $L^{2}(\Omega)$
 	and the following holds.
 	\begin{itemize}
 		\item[i.] (Existence and Uniqueness)
 		There exists $T>0$ independent of $\nu$ and a unique solution
 		$u^\nu \in L^\infty(0,T; L^2(\Omega) \cap W^{3,3+\delta}_\cco(\Omega) \cap W^{3,6+2\delta}_\cco(\Omega)$
 		of \eqref{NSE0}--\eqref{hnavierbdry} on~$[0,T]$ with the initial data $u_0^\nu$.
 		Moreover, there is $M>0$ depending only on the size of $u_0$ such that
 		\begin{align}
 			\begin{split}
 				\sup_{[0,T]}&
 				(\Vert u^\nu(t)\Vert_{L^2}^2
 				+\Vert u^\nu(t)\Vert_{3,3+\delta}^{3+\delta}
 				+\Vert u^\nu(t)\Vert_{3,6+2\delta}^{6+2\delta}
 				+\Vert \nabla u^\nu(t)\Vert_{2,3+\delta}^{3+\delta}
 				+\Vert \nabla u^\nu(t)\Vert_{L^\infty})
 				\le M.
 				\label{EQ.main2}
 			\end{split}
 		\end{align}
 		\item[ii.] (Inviscid Limit)
 		There exists  a unique solution 
 		$u \in L^\infty(0,T; L^{2}(\Omega) \cap W^{3,3+\delta}_\cco(\Omega)\cap W^{3,6+2\delta}_\cco(\Omega)$
 		with 
 		$\nabla u \in L^\infty(0,T;W^{2,3+\delta}_\cco(\Omega)\cap L^\infty(\Omega))$
 		to the Euler equations
 		\eqref{euler}
 		such that 
 		\begin{align}
 			\sup_{[0,T]}
 			\Vert u^\nu -u\Vert_{L^p}
 			\le \bar{M}\nu^\frac{3+p}{5p}
 			\comma p \in [2,\infty],
 			\label{invlimit1}
 		\end{align}
 		where $\bar{M}>0$ is independent of~$\nu$.
 	\end{itemize}
 \end{Theorem}
 \colb
 
 Assuming that $\mu\ge 0$, we establish
 the existence, uniqueness, and the inviscid limit with 
 one less conormal derivative on $\nabla u$.
 
 \cole
 \begin{Theorem}[A sharper result for non-negative friction]
  \label{T03}
  Assume that $\nu \in (0,\bar{\nu}]$ and $\mu \ge 0$.
  Let $u_0 \in L^2(\Omega) \cap W^{3,6}_\cco(\Omega) \cap W^{2,\infty}_\cco(\Omega)$
  be divergence-free with vanishing normal component on~$\partial \Omega$,
  and $\nabla u_0 \in W^{1,6}_\cco(\Omega) \times L^\infty(\Omega)$ satisfy
   $\partial_z (u_0)_\hh = 2\mu (u_0)_\hh$ on~$\partial \Omega$.
   Then, there exists a sequence of smooth divergence-free initial data $u_0^\nu$
   such that $u_0^\nu \to u_0$ in $L^2(\Omega)$
   and the following holds.
  \begin{itemize}
   \item[i.] (Existence and Uniqueness)
   There exists $T>0$ independent of $\nu$ and a unique solution
   $u^\nu \in C([0,T];L^2(\Omega) \cap W^{1,6}_\cco(\Omega)) \cap L^\infty(0,T;W^{3,6}_\cco(\Omega)\cap W^{2,\infty}_\cco(\Omega))$
   to \eqref{NSE0}--\eqref{hnavierbdry} on~$[0,T]$ with initial data $u_0^\nu$.
 Moreover, there exists $M>0$ depending only on the size of $u_0$ such that
   \begin{align}
    \begin{split}
    \sup_{[0,T]}&
    (\Vert u^\nu(t)\Vert_{L^2}^{2}
    +\Vert u^\nu(t)\Vert_{3,6}^{6}
    +\Vert u^\nu(t)\Vert_{2,\infty}^{2}
    +\Vert \nabla u^\nu(t)\Vert_{1,3+\delta}^{3+\delta}
    +\Vert \nabla u^\nu(t)\Vert_{L^\infty}^2)
    \le M.
    \llabel{EQ.main2}
    \end{split}
   \end{align} 
 \item[ii.] (Inviscid Limit)
   There exists a unique solution 
   $u \in L^\infty(0,T;L^2(\Omega)\cap W^{3,6}_\cco(\Omega)\cap W^{2,\infty}_\cco(\Omega))$
   with 
   $\nabla u \in L^\infty(0,T;W^{1,6}_\cco(\Omega)\cap L^\infty(\Omega))$
   to the Euler equations
   \eqref{euler}
   such that 
   \eqref{invlimit1} holds.
\end{itemize}
 \end{Theorem}
 \colb
 
 Theorem~\ref{T03} establishes the inviscid limit under weaker assumptions on the initial data.
 Indeed, when $f,\nabla f \in W^{2,3+\delta}_\cco$,
 we have the inequality
 \begin{align}
 	\Vert f\Vert_{2,\infty} \lec \Vert \nabla f\Vert_{2,3+\delta}+\Vert f\Vert_{2,3+\delta}
 	;\llabel{demo}
 \end{align}
 see~\cite{MR2}.
 Therefore, the class of initial data considered in Theorem~\ref{T01} contains $W^{2,\infty}_\cco$.
 Theorem~\ref{T03} holds under weaker assumptions due to the favorable sign of
 the friction coefficient $\mu$. 
 When $\mu \neq 0$, the Navier boundary condition is a Robin-type boundary condition
 and when $\mu < 0$, it is more challenging to close the estimates.
 In fact, assuming that $\mu \ge 0$, \cite[Proposition~6.1]{AK1} propagates $\Vert u\Vert_{2,\infty}$
 uniformly in time and viscosity, even in the presence of the Laplacian. However, it is unclear whether 
 the same result holds when $\mu<0$. 
 
 In our inviscid limit results, we impose conormal differentiability on $\nabla u_0$ to control the
 boundary or commutator terms resulting from the Laplacian. 
 However, when we only consider the Euler equations, it is possible to improve 
 upon the assumptions on $u_0$ and construct unique solutions.
 
 \cole
 \begin{Theorem}[Well-posedness of the Euler equations in Sobolev conormal spaces]
  \label{T04}
  Let $u_0 \in L^2(\Omega) \cap W^{3,3+\delta}_\cco(\Omega) \cap W^{2,\infty}_\cco(\Omega)$,
  with $\nabla u_0 \in L^\infty(\Omega)$, be
  such that $\div u_0 = 0$ and $u_0 \cdot n= 0$ on~$\partial \Omega$.
For some $T>0$,
   there exists  a unique solution 
   $u \in L^\infty(0,T;L^2(\Omega)\cap W^{3,3+\delta}_\cco(\Omega)\cap W^{2,\infty}_\cco(\Omega))$
   with $\nabla u \in L^\infty(\Omega \times (0,T))$
   to the Euler equations
   \eqref{euler}
   \begin{align}
    \sup_{[0,T]}
    (\Vert u(t)\Vert_{L^2}^2
    +\Vert u(t)\Vert_{3,3+\delta}^{3+\delta}
    +\Vert u(t)\Vert_{2,\infty}
    +\Vert \nabla u(t)\Vert_{L^\infty})
    \le M,
    \llabel{EQ.main3}
   \end{align}
where $M>0$ depends on the norms of the initial data. 
 \end{Theorem}
 \colb

We note that in Theorems~\ref{T01}--\ref{T04} we can replace 
$\nabla u_0$ by $\curl u_0$. In addition, 
when we restrict attention to bounded flat domains,
we may omit 
$L^{3+\delta}$-based conormal differentiability assumptions on $u_0$ in Theorem~\ref{T01}
as well as $L^2$ integrability assumptions in all of the results.
 
 When we consider conormal derivatives,
 we distinguish between $Z_\hh$ and $Z_3$
 upon letting
 \begin{align}
 	Z^\alpha = Z^{\tilde{\alpha}}_\hh Z^k_3
 	\comma \alpha = (\tilde{\alpha}, k) \in \mathbb{N}_0^2 \times \mathbb{N}_0
 	.
 	\llabel{ZhZ3}  
 \end{align} 
 In addition, we compute the commutator of $Z_3$ and $\partial_z$
 by using the following lemma.
 
 \cole
 \begin{lemma}
 	\label{L01}
 	Let $f$ be a smooth function. Then there exist
 	$\{c^k_{j,\varphi}\}_{j=0}^k$ and $\{\tilde{c}^k_{j,\varphi}\}_{j=0}^k$ 
 	smooth, bounded functions
 	 of $z$, for $k \in \mathbb{N}$, depending on $\varphi$ such that
 	\begin{align}
 		\begin{split}    
 			&(i)\, Z^k_3 \partial_z f 
 			=
 			\sum_{j=0}^{k}
 			c^k_{j,\varphi} \partial_z Z^j_3 f
 			=\partial_z Z_3^kf + \sum_{j=0}^{k-1}
 			c^k_{j,\varphi} \partial_z Z^j_3 f
 			,
 			\\
 			&(ii)\, \partial_z Z^k_3 f 
 			=
 			\sum_{j=0}^{k}
 			\tilde{c}^k_{j,\varphi} Z^j_3 \partial_z f
 			=Z_3^k \partial_z f +\sum_{j=0}^{k-1}
 			\tilde{c}^k_{j,\varphi} Z^j_3 \partial_z f
 			,
 			\\
 			&(iii)\, Z^k_3 \partial_{zz} f 
 			=
 			\sum_{j=0}^{k}
 			\sum_{l=0}^{j}
 			\left(
 			c^j_{l,\varphi} c^k_{j,\varphi} \partial_{zz} Z^l_3 f 
 			+
 			(c^j_{l,\varphi})' c^k_{j,\varphi} \partial_{z} Z^l_3 f 
 			\right)
 			,
 			\\
 			&(iv)\, \partial_{zz} Z^k_3 f 
 			=
 			\sum_{j=0}^{k}
 			\sum_{l=0}^{j}
 			\tilde{c}^j_{l,\varphi} \tilde{c}^k_{j,\varphi} Z^l_3 \partial_{zz} f 
 			+
 			\sum_{j=0}^{k}
 			(\tilde{c}^k_{j,\varphi})' Z^j_3 \partial_z f 
 			,
 			\label{EQL02}
 		\end{split}
 	\end{align}
 	where $\tilde{c}^k_{k,\varphi} = 1 = c^k_{k,\varphi}$, and
 	the prime indicates the derivative with respect to the variable~$z$.
 \end{lemma}
 \colb
 
 We also utilize an interpolation type inequality for conormal derivatives.
 
 \begin{Lemma}[An interpolation inequality for Sobolev conormal spaces]
 	\label{L03}
 	Let $2 \le p < \infty$ and assume that $f,g\in L^\infty(\Omega) \cap W^{k,p}_\cco(\Omega)$, for $k \in \mathbb{N}$. Then
 	\begin{align}
 		\Vert Z^{\alpha}f 
 		Z^{\beta}g\Vert_{L^p}
 		\lec
 		\Vert f\Vert_{L^\infty}
 		\Vert g\Vert_{k,p}
 		+
 		\Vert f\Vert_{k,p}
 		\Vert g\Vert_{L^\infty}
 		,\label{EQ.int}
 	\end{align}
 	for any $\alpha$, $\beta \in \mathbb{N}^3_0$ with $|\alpha|+|\beta|=k$. 
 \end{Lemma}
 \colb
 
 The proof of \eqref{EQ.int} follows from a Gagliardo-Nirenberg type
 inequality
 \begin{align}
 	\Vert f\Vert_{|\alpha|,\frac{kp}{|\alpha|}}
 	 \lec
 	  \Vert f\Vert_{L^{\infty}}^\frac{|\beta|}{k}\Vert f\Vert_{k,p}^\frac{|\alpha|}{k}
 	   ,\label{EQgag}
 \end{align}
 on $\mathbb{R}^3$ employed with an even extension along the vertical variable;
 see~\cite{Gu}.
 Next, since $u_3 = 0$ on $\partial \Omega$, the Hardy inequality and the 
 incompressibility condition imply
 \begin{align}
  \left\Vert \frac{u^\nu_3}{\varphi}\right\Vert_{k,p}
  \lec \Vert Z_\hh u^\nu_\hh\Vert_{k,p}
  ,
  \label{EQ.u3}
 \end{align}
 for all $p \in [1,\infty]$.
 
 Now, we fix $\nu \in (0,\bar{\nu}]$,
 and denote by $(u,p)$, instead of $(u^\nu,p^\nu)$, the smooth solution to the Navier~Stokes system
 with the viscosity~$\nu$.
 Also, we write
  \begin{align}
  	\begin{split}
  	P &= P (\Vert u\Vert_{L^{2}},\Vert u\Vert_{3,3+\delta},\Vert u\Vert_{3,6+2\delta},\Vert \eta\Vert_{2,3+\delta},\Vert \eta\Vert_{L^{\infty}})
  	\\
 	Q &= Q (\Vert u\Vert_{L^{2}},\Vert u\Vert_{3,6},\Vert \eta\Vert_{1,6},\Vert u\Vert_{2,\infty},\Vert \eta\Vert_{L^{\infty}})
 	\\
 	R  &= R (\Vert u\Vert_{L^{2}},\Vert u\Vert_{3,3+\delta},\Vert u\Vert_{2,\infty},\Vert \omega\Vert_{L^\infty})
 	,\label{pol1}
 	\end{split}
 \end{align}
 for polynomials that may change from line to line, and we write $(P,Q,R)|_{t=0} = (P_0,Q_0,R_0)$.
 We also define
 \begin{align}
 	\mathcal{M}_{0,j,p}(u_0)=\mathcal{M}_{0,j,p}
 	=\sum_{0\le |\alpha|\le j}
 	\left(\nu^{\frac{p-1}{p}}[\nabla Z^\alpha u_0]_{1-\frac{2}{p},p,x,\Omega}
 	+\nu^{\frac{1}{2}}\Vert \nabla Z^\alpha u_0\Vert_{L^{p}}
 	+\Vert Z^\alpha u_0\Vert_{L^{p}}
 	\right)
 	,\llabel{u.ini}
 \end{align}
 and
 \begin{align}
 	\mathcal{N}_{0,j,p}(\eta_0)
 	=
 	\mathcal{N}_{0,j,p}
 	=
 	\sum_{0\le |\alpha|\le j}
 	\left(\nu^{\frac{p-1}{p}}[\nabla Z^\alpha \eta_0]_{1-\frac{2}{p},p,x,\Omega}
 	+\nu^{\frac{1}{2}}\Vert \nabla Z^\alpha \eta_0\Vert_{L^{p}}
 	+\Vert Z^\alpha \eta_0\Vert_{L^{p}}
 	\right)
 	,\llabel{eta.ini}
 \end{align}
 for $j \in \mathbb{N}_0$ and $p >2$; see Section~\ref{sec.max}
 for the definitions of the seminorms $[\cdot]_{s,p,x,\Omega}$. 
 In the rest of this work, we refer to the solutions of the Euler equations
 with slip boundary conditions as the solutions $u$ of \eqref{NSE0}--\eqref{hnavierbdry} with
 $\nu = 0$.
 Now, we state the a~priori estimates.
 
 \cole 
 \begin{Proposition}[A priori estimates]
  \label{P.Ap}
  Let $\nu \in [0,\bar{\nu}]$, $\omega = \curl u$, and 
  $\eta = \omega_\hh - 2\mu u_\hh^\perp$ (see~\eqref{eta}). 
  Then, any smooth solution $u$
  to \eqref{NSE0}--\eqref{hnavierbdry} defined on $[0,T]$
  with a smooth initial datum $u_0$,
  satisfies the following:
  \begin{itemize}
   \item[i.] If $\mu \in \mathbb{R}$ and $\nu >0$, we have
   \begin{align}
   	\begin{split}
   		\Vert u&\Vert_{L^{2}}^2
    +\Vert u\Vert_{3,3+\delta}^{3+\delta}
     +\Vert u\Vert_{3,6+2\delta}^{6+2\delta}
      +\Vert \eta\Vert_{2,3+\delta}^{3+\delta}
       +\Vert \eta\Vert_{L^{\infty}}
    \\& \le C\left(P_0 
     +\mathcal{M}_{0,2,6+2\delta}^{6+2\delta}
     +\mathcal{N}_{0,1,3+\delta}^{3+\delta}
      + \int_0^t P\,ds\right),
    \label{ap1}
    \end{split}
   \end{align}
   for $t \in [0,T]$.
   \item[ii.] If $\mu \ge 0$ and $\nu >0$, we have
   \begin{align}
   	\begin{split}
   		\Vert u\Vert_{L^{2}}^2
   		+\Vert u\Vert_{3,6}^{6}
   		+\Vert \eta\Vert_{1,6}^{6}
   		+\Vert u\Vert_{2,\infty}^2
   		+\Vert \eta\Vert_{L^{\infty}}^2
   		&\le C\left(Q_0  
   		+\mathcal{M}_{0,2,6}^6
   		+\mathcal{N}_{0,0,6}^6
   		 +\int_0^t Q\,ds\right),
   		\label{ap2}
   	\end{split}
   \end{align}
for $t \in [0,T]$.
   \item[iii.] If $\nu =0$, we have
   \begin{align}
   	\begin{split}
   		\Vert u\Vert_{L^{2}}^2
   		+\Vert u\Vert_{3,3+\delta}^{3+\delta}
   		+\Vert u\Vert_{2,\infty}
   		+\Vert \omega\Vert_{L^\infty}
   		\le C\left(R_0+ 
   		 \int_0^t R\,ds\right),
   		\label{ap3}
   	\end{split}
   \end{align}
   for $t \in [0,T]$.
\end{itemize}
\end{Proposition}
\colb
 
 To establish the a~priori bounds, 
 we prove estimates for the conormal derivatives of $u$, $\nabla u$,
 and~$\nabla p$
 in Sections~\ref{sec.co}, \ref{sec.no},
 and \ref{sec.p}, respectively. 
 Next, we establish $L^\infty$ bounds for $u$ and $\nabla u$ in Section~\ref{sec.inf}. 
 Following this, we present maximum regularity and trace estimates in Sections~\ref{sec.max} and~\ref{sec.tr}, respectively.
 In Section~\ref{sec.apri}, we conclude the proof of the
 a~priori bounds. Finally, in Section~\ref{sec.main}, we 
 prove Theorems~\ref{T01}--\ref{T04}.
 
 \startnewsection{Conormal Derivative Estimates}{sec.co}
 
 In this section, we present the conormal derivative estimates for $u$. 
 We recall the convention that the solutions of
 \eqref{NSE0} and \eqref{hnavierbdry} with $\nu=0$
 refer to the solutions for the Euler equations \eqref{euler} and \eqref{eulerb}.
 
 \cole
 \begin{Proposition}
  \label{P.Con}
  Let $\mu \in \mathbb{R}$, $\nu \in [0,\bar{\nu}]$, $p \in (2,\infty)$, and assume that $u$ is a smooth solution of \eqref{NSE0} and \eqref{hnavierbdry} on~$[0,T]$ with a smooth initial datum $u_0$.
  Then we have the inequality
  \begin{align}
  	\begin{split}
  		\Vert  u(t)&\Vert_{3,p}^{p}
  		+c_0\nu \sum_{0\le |\alpha| \le 3}\left(\int_0^t \int_\Omega \left(|\nabla Z^\alpha u|^2 |Z^\alpha u|^{p-2}
  		+ |\nabla |Z^\alpha u|^\frac{p}{2}|^2\right) \,dxds\right)
  		\\&\lec
  		\Vert u_0\Vert_{3,p}^p
  		+\int_0^t\left(
  		\Vert u\Vert_{3,p}^{p}
  		(\Vert u\Vert_{2,\infty}+\Vert u\Vert_{W^{1,\infty}})
  		+\Vert u\Vert_{3,p}^{p-1}\Vert \nabla p\Vert_{3,p}\right)\,ds
  		+\nu^{\frac{p}{2}}\int_0^t \Vert \partial_z u\Vert_{2,p}^p\,ds
  		.
  		\label{EQ.Con}
  	\end{split}
  \end{align}
where $c_0>0$ and $t \in [0,T]$. 
 \end{Proposition}
 \colb
 
 \begin{proof}[Proof of Proposition~\ref{P.Con}]
  We establish \eqref{EQ.Con} by induction on
  the order of conormal differentiability.
  The base case is the standard $L^p$
  estimate given by
   \begin{align}
   \frac{1}{p}\frac{d}{dt}\Vert u\Vert_{L^{p}}^{p}
   +\nu\int_{\Omega} |\nabla u|^2 |u|^{p-2} \,dx
   +4\nu\frac{p-2}{p^2} \int_{\Omega} |\nabla |u|^\frac{p}{2} |^2 \,dx
   = -2\mu \nu \Vert u_h\Vert_{L^{p}(\partial \Omega)}^{p}
   -\int_{\Omega} \nabla p u |u|^{p-2} \,dx
   ,\label{L^p}
  \end{align}
  where we have used \eqref{hnavierbdry} to get  
\begin{align}
	\nu\int_{\partial \Omega} \partial_i u_i u_j |u|^{p-2} n_i
	=\nu\int_{\partial \Omega} \partial_z u_\hh \cdot u_\hh |u_\hh|^{p-2}
	=-2\mu \nu \Vert u_\hh\Vert_{L^{p}(\partial \Omega)}^{p}
	. \llabel{boundary}
\end{align}
  When $\mu < 0$, we estimate this boundary term by writing
  \begin{align}
  	|\mu| \nu \Vert u_\hh\Vert_{L^{p}(\partial \Omega)}^{p}
  	 \le 
  	  \mu \nu \Vert |u_\hh|^{\frac{p}{2}}\Vert_{L^{2}(\partial \Omega)}^{2}
  	   \lec
  	    \mu \nu \Vert \nabla |u_\hh|^{\frac{p}{2}}\Vert_{L^{2}}\Vert |u_\hh|^{\frac{p}{2}}\Vert_{L^{2}}
  ,\llabel{boundary2}
  \end{align}
  from where we invoke Young's inequality and obtain
  \begin{align}
  	|\mu| \nu \Vert u_\hh\Vert_{L^{p}(\partial \Omega)}^{p}
  	 \le 
  	  \epsilon \nu \Vert \nabla |u_\hh|^{\frac{p}{2}}\Vert_{L^{2}}^2
  	   +C_\epsilon\Vert |u_\hh|^{\frac{p}{2}}\Vert_{L^{2}}^2,
  	   \label{boundary3}
  \end{align}
  where $\epsilon>0$ is sufficiently small.
  Combining \eqref{L^p} and \eqref{boundary3} and letting $ t \in [0,T]$
   yields
   \begin{align}
   	\Vert u(t)\Vert_{L^{p}}^{p}
   	+c_0\nu \int_0^t \int_\Omega \bigl(|\nabla u|^2 |u|^{p-2}
   	+ |\nabla |u|^\frac{p}{2} |^2\bigr) \,dxds
   	\lec \Vert u_0\Vert_{L^{p}}^p +\int_0^t \Vert u\Vert_{L^{p}}^{p-1}
   	 (\Vert u\Vert_{L^{p}} + \Vert \nabla p\Vert_{L^{p}})\,ds
   	,\llabel{base}
   \end{align}
   and this concludes the base step of the induction.
  For the remaining part, we only 
  present the final step. Therefore,
  for $|\alpha'|\le 2 $, we assume that 
  \begin{align}
  	\begin{split}
  	\Vert Z^{\alpha'} &u(t)\Vert_{L^{p}}^{p}
  	+c_0\nu \int_0^t \int_{\Omega} \bigl(|\nabla Z^{\alpha'} u|^2 |Z^{\alpha'} u|^{p-2}
  	+ |\nabla |Z^{\alpha'} u|^\frac{p}{2} |^2\bigr) \,dxds
  	\\&\lec \Vert u_0\Vert_{2,p}^p +\int_0^t \bigl(
  	\Vert u\Vert_{2,p}^p 
  	(\Vert u\Vert_{W^{1,\infty}}+
  	\Vert u\Vert_{2,\infty}+1
  	)+\Vert u\Vert_{2,p}^{p-1}\Vert \nabla p\Vert_{2,p} \bigr) \,ds
  		+\nu^{\frac{p}{2}}\int_0^t \Vert \nabla u\Vert_{2,p}^p\,ds
   , \label{base2}
  	\end{split}
  \end{align}
  and aim to establish \eqref{EQ.Con}.
  
  Let $Z^\alpha$ be a conormal derivative of order three, i.e., $|\alpha|=3$,
  and first assume that $Z^\alpha = Z_\hh^{\tilde{\alpha}}$.
  This corresponds to the case where all derivatives are horizontal.
  Recalling that $Z_\hh$ commutes with $\partial_i$, for $i=1,2,3$,
  we apply $Z_\hh^{\tilde{\alpha}}$ to \eqref{NSE0} and write
  \begin{align}
   \partial_t Z^{\tilde{\alpha}}_\hh u
   -\nu \Delta Z^{\tilde{\alpha}}_\hh u
   + u \cdot \nabla Z^{\tilde{\alpha}}_\hh u
   + \nabla Z^{\tilde{\alpha}}_\hh p
   =
   u \cdot \nabla Z^{\tilde{\alpha}}_\hh u
   -
   Z^{\tilde{\alpha}}_\hh(u \cdot \nabla u)
   .
   \llabel{EQ01}
  \end{align}
  Testing this with $Z^{\tilde{\alpha}}_\hh u|Z^{\tilde{\alpha}}_\hh u|^{p-2}$ yields
  \begin{align}
  	\begin{split}
   \frac{1}{p}\frac{d}{dt}&\Vert Z^{\tilde{\alpha}}_\hh u\Vert_{L^{p}}^{p}
   +\nu\int_{\Omega} |\nabla Z^{\tilde{\alpha}}_\hh u|^2 |\nabla Z^{\tilde{\alpha}}_\hh u|^{p-2} \,dx
   +4\nu\frac{p-2}{p^2} \int_{\Omega} |\nabla |Z^{\tilde{\alpha}}_\hh u|^\frac{p}{2}|^2 \,dx
   \\&=
   -2\mu \nu \Vert \nabla Z^{\tilde{\alpha}}_\hh u_\hh\Vert_{L^{p}(\partial \Omega)}^{p}
   +
   \int_{\Omega} 
    \left(u \cdot \nabla Z^{\tilde{\alpha}}_\hh u
    -Z^{\tilde{\alpha}}_\hh(u \cdot \nabla u)\right)
      Z^{\tilde{\alpha}}_\hh u |\nabla Z^{\tilde{\alpha}}_\hh u|^{p-2}\,dx
      \\&\indeq
   -\int_{\Omega} Z_\hh^{\tilde{\alpha}} \nabla p,Z^{\tilde{\alpha}}_\hh u |\nabla Z^{\tilde{\alpha}}_\hh u|^{p-2} \,dx
   ,
   \label{EQ54}
   \end{split}
  \end{align}
  where we have used~\eqref{hnavierbdry}.
  Now, we repeat the steps leading to \eqref{boundary3} for $Z_\hh^{\tilde{\alpha}} u_\hh$
  and obtain
  \begin{align}
   |\mu| \nu \Vert Z^{\tilde{\alpha}}_\hh u\Vert_{L^{p}(\partial \Omega)}^{p}
   \le
   \epsilon \nu \Vert \nabla |Z^{\tilde{\alpha}}_\hh u|^{\frac{p}{2}}\Vert_{L^2}^2
   +C_\epsilon\Vert u\Vert_{3,p}^{p}
   ,\llabel{EQ144}
  \end{align}
  where $\epsilon>0$ is sufficiently small.
  Next, we use H\"older inequality to estimate
  the remaining terms on the right-hand side of \eqref{EQ54}.
  It follows that
  \begin{align}
     \int_{\Omega} Z_\hh^{\tilde{\alpha}} \nabla p Z^{\tilde{\alpha}}_\hh u |\nabla Z^{\tilde{\alpha}}_\hh u|^{p-2}\,dx
     \lec
      \Vert \nabla p\Vert_{3,p}\Vert Z_\hh^{\tilde{\alpha}}u\Vert_{L^{p}}^{p-1}	
     ,\llabel{conpre1}
  \end{align}
  and
  \begin{align}
     \int_{\Omega} \left(u \cdot \nabla Z^{\tilde{\alpha}}_\hh u
     -Z^{\tilde{\alpha}}_\hh(u \cdot \nabla u)\right)
      Z^{\tilde{\alpha}}_\hh u |\nabla Z^{\tilde{\alpha}}_\hh u|^{p-2} \,dx
      \lec
      \Vert u \cdot \nabla Z^{\tilde{\alpha}}_\hh u-Z^{\tilde{\alpha}}_\hh(u \cdot \nabla u)\Vert_{L^{p}}
       \Vert Z_\hh^{\tilde{\alpha}}u\Vert_{L^{p}}^{p-1}	
      . \label{concom1}	
  \end{align}
  We expand the first term on the right-hand side of \eqref{concom1} as
  \begin{align}
   u \cdot \nabla Z^{\tilde{\alpha}}_\hh u
   -
   Z^{\tilde{\alpha}}_\hh(u \cdot \nabla u)
   =
   -\sum_{1\le |\tilde{\beta}| \le 3} 
   {\tilde{\alpha} \choose \tilde{\beta}}
   (Z^{\tilde{\beta}}_\hh u_\hh 
   \cdot 
   \nabla_\hh Z^{\tilde{\alpha}-\tilde{\beta}}_\hh u
   + 
   Z^{\tilde{\beta}}_\hh u_3 
   \partial_z Z^{\tilde{\alpha}-\tilde{\beta}}_\hh u)
   .
   \llabel{EQ45}                              
  \end{align}
  Note that
  \begin{align}
   \Vert 
   Z^{\tilde{\beta}}_\hh u_\hh 
   \cdot 
   \nabla_\hh Z^{\tilde{\alpha}-\tilde{\beta}}_\hh u
   \Vert_{L^{p}}
   \lec
   \Vert u\Vert_{2,\infty}\Vert u\Vert_{3,p}
   \comma 1 \le |\tilde{\beta}| \le 3,
   \label{EQ30}
  \end{align} 
  Proceeding to the term involving $\partial_z$, 
  we insert $\varphi \frac{1}{\varphi}$, obtaining
  \begin{equation}
   \label{EQ143}
   \left\Vert Z_\hh^{\tilde{\beta}} \frac{u_3}{\varphi} Z_3 Z_\hh^{\tilde{\alpha} - {\tilde{\beta}}} u\right\Vert_{L^{p}}  
   \lec 
   \begin{cases}
    \bigl\Vert Z \frac{u_3}{\varphi}\bigr\Vert_{L^\infty}\Vert u\Vert_{3,p}
    \lec \Vert u\Vert_{2,\infty}\Vert u\Vert_{3,p}, & |{\tilde{\beta}}|=1 \\
    \bigl\Vert Z_\hh \frac{u_3}{\varphi}\bigr\Vert_{1,p}\Vert Z_\hh Z_3 u\Vert_{L^\infty}
    \lec \Vert u\Vert_{3,p}\Vert u\Vert_{2,\infty}, & |{\tilde{\beta}}|=2,
   \end{cases}
  \end{equation}
  using Hardy's inequality and the divergence-free condition.
  To conclude, note that 
  \begin{align}
   \Vert Z_\hh^{\tilde{\beta}} u_3 \partial_z Z_\hh^{\tilde{\alpha} - {\tilde{\beta}}} u\Vert_{L^{p}}  
   \lec 
   \Vert u\Vert_{3,p}\Vert \partial_z u\Vert_{L^\infty}
   \comma   |{\tilde{\beta}}|=3
   .\label{EQ33} 
  \end{align}
  Now, we collect \eqref{EQ54}--\eqref{EQ33} and integrate in time 
  so that we obtain
  \begin{align}
  	\begin{split}
   \Vert Z^{\tilde{\alpha}}_\hh u(t)&\Vert_{L^{p}}^{p}
   +
   c_0\nu\int_0^t \int_{\Omega}
    \left((\nabla |Z^{\tilde{\alpha}}_\hh u|^{\frac{p}{2}})^2 
   +
    |\nabla Z^{\tilde{\alpha}}_\hh u|^2 |\nabla Z^{\tilde{\alpha}}_\hh u|^{p-2}
   \right)\,dxds
   \\&\lec
   \Vert u_0\Vert_{3,p}^{p}
   +
   \int_0^t\left(
   \Vert u\Vert_{3,p}^{p}
   (\Vert u\Vert_{2,\infty}+\Vert u\Vert_{W^{1,\infty}})
   +\Vert u\Vert_{3,p}^{p-1}\Vert \nabla p\Vert_{3,p}\right)\,ds
   ,  
   \llabel{EQ38}
   \end{split}   
  \end{align}
  for~$t \in [0,T]$.
    
  Now, let $Z^\alpha = Z^{\tilde{\alpha}}_\hh Z^k_3$ 
  where $1\le k \le 3$. It follows that $Z^\alpha u$ solves
  \begin{align}
   Z^\alpha u_t
   - \nu \Delta Z^\alpha u
   + u\cdot \nabla Z^\alpha u
   = 
   u\cdot \nabla Z^\alpha u - Z^\alpha(u\cdot \nabla u) 
   - Z^\alpha \nabla p
   + \nu Z^\alpha \Delta u
   - \nu \Delta Z^\alpha u
   .
   \label{EQ46}
  \end{align}
  We multiply this equation by $Z^\alpha u |Z^\alpha u|^{p-2}$, 
  the left-hand side of \eqref{EQ46} gives
  \begin{align}
  	\begin{split}
  		\frac{1}{p}&\frac{d}{dt}\Vert Z^\alpha u\Vert_{L^{p}}^{p}
  		+\nu\int_{\Omega} |\nabla Z^\alpha u|^2 |Z^\alpha u|^{p-2} \,dx
  		+4\nu\frac{p-2}{p^2} \int_{\Omega} |\nabla |Z^\alpha u|^\frac{p}{2}|^2 \,dx
  		\\&=
  		\int_{\Omega} \left(u \cdot \nabla Z^\alpha u
  		-Z^\alpha(u \cdot \nabla u)\right)
  		Z^\alpha u |\nabla Z^\alpha u|^{p-2} \,dx
  		-\int_\Omega Z^\alpha \nabla p Z^\alpha u |\nabla Z^\alpha u|^{p-2} \,dx
  		\\&\indeq
  		+\nu \int_{\Omega} (Z^\alpha \Delta u
  		- \Delta Z^\alpha u) Z^\alpha u |\nabla Z^\alpha u|^{p-2}\,dx
  		.
  		\label{conork}
  	\end{split}
  \end{align}
  We note that \eqref{conork} does not have a boundary term since 
  $Z_3 = \varphi \partial_z = 0$ on~$\partial \Omega$.
  Now, we rewrite the quadratic commutator term as 
  \begin{align}
   \begin{split}
    u \cdot \nabla Z^\alpha u &- Z^\alpha(u \cdot \nabla u)
    =
    u \cdot \nabla Z^\alpha u - u \cdot Z^\alpha \nabla u
    -
    \sum_{1 \le |\beta|\le |\alpha|}
    {\alpha \choose \beta}
    Z^\beta u \cdot Z^{\alpha-\beta} \nabla u
    \\&=
    u_3 \partial_z Z^\alpha u - u_3 Z^\alpha \partial_z u    
    -
    \sum_{1\le |\beta|\le |\alpha|}
    {\alpha \choose \beta}
    Z^\beta u \cdot Z^{\alpha-\beta} \nabla u
    \\&=
    -
    \sum_{j=0}^{k-1}
    \tilde{c}^k_{j,\varphi} u_3 \partial_z Z^{\tilde{\alpha}}_\hh Z^j_3 u  
    -
    \sum_{1\le |\beta|\le |\alpha|}
    {\alpha \choose \beta}
    Z^\beta u \cdot Z^{\alpha-\beta} \nabla u
    = I_1 + I_2
    ,
    \llabel{EQ13}        
   \end{split}
  \end{align}
  recalling that $|\alpha|=3$.
  Next, we multiply $I_1$ by $\varphi \frac{1}{\varphi}$ and write
  \begin{align}
   \sum_{j=0}^{k-1}
   \left\Vert \tilde{c}^k_{j,\varphi} \frac{u_3}{\varphi} Z^{\tilde{\alpha}}_\hh Z^{j+1}_3 u\right\Vert_{L^{p}}  
   \lec
   \left\Vert \frac{u_3}{\varphi}\right\Vert_{L^\infty}\Vert u\Vert_{3,p}
   \lec
   \Vert u\Vert_{1,\infty}\Vert u\Vert_{3,p}
   ,
   \label{EQ56}
  \end{align}
  where we have	used~\eqref{EQ.u3}.
  Proceeding to $I_2$, we have  
  \begin{align}
   I_2 = -
   \sum_{1\le |\beta| \le |\alpha|}
   {\alpha \choose \beta} 
   (Z^{\beta} u_\hh 
   \cdot 
   \nabla_\hh Z^{\alpha-\beta} u
   + 
   Z^{\beta} u_3 
   Z^{\alpha-\beta} \partial_z u)
   = I_{21} + I_{22}
   .
   \llabel{EQ57}
  \end{align}
  We estimate $I_{21}$
  by utilizing the bounds in~\eqref{EQ30},
  while for $I_{22}$, we employ \eqref{EQ33}
  when~$|\beta|=3$ and  Lemma~\ref{L01}(i) otherwise. 
  For the commutator terms resulting from 
  \eqref{EQL02}$_1$,
  we note that
  \begin{align}
   \frac{Z_3 u_3}{\varphi} = \nabla_\hh \cdot u_\hh
   \text{ and }
   \frac{1}{\varphi} Z_\hh = Z_\hh \frac{1}{\varphi}
   ,
   \llabel{EQ32}
  \end{align}
  so that we may proceed as in \eqref{EQ143}.
  It follows that
  \begin{align}
   \int_{\Omega} \left(u \cdot \nabla Z^\alpha u
   -Z^\alpha(u \cdot \nabla u)\right)
   Z^\alpha u |\nabla Z^\alpha u|^{p-2}) \,dx
   \lec 
   \Vert u\Vert_{3,p}^{p}
   (\Vert u\Vert_{2,\infty}+\Vert u\Vert_{W^{1,\infty}}).
   \label{EQ58}
  \end{align}
  Recalling \eqref{EQ46}, we now
  bound the pressure term by writing
  \begin{align}
   \int_{\Omega} Z^\alpha \nabla p Z^\alpha u|Z^\alpha u|^{p-2} \,dx
   \lec
   \Vert \nabla p\Vert_{3,p}\Vert u\Vert_{3,p}^{p-1}
   .
   \label{EQ47}
  \end{align}
  
  It only remains to estimate the commutator term involving the Laplacian. 
  Employing Lemma~\ref{L01}, we have 
  \begin{align}
   \begin{split}
    \nu Z^\alpha \Delta u
    - \nu \Delta Z^\alpha u
    =&
    \nu Z^\alpha \partial_{zz} u - \nu \partial_{zz} Z^\alpha u
    =
    \nu Z^{\tilde{\alpha}}_\hh (Z_3^k \partial_{zz} u - \partial_{zz} Z_3^k u)
    \\
    =&
    \nu
    \sum_{j=0}^{k-1}
    \sum_{l=0}^{j}
    \left(
    c^j_{l,\varphi} c^k_{j,\varphi} \partial_{zz} Z^l_3 Z^{\tilde{\alpha}}_\hh u 
    +
    (c^j_{l,\varphi})' c^k_{j,\varphi} \partial_{z} Z^l_3 Z^{\tilde{\alpha}}_\hh u
    \right)
    \\
    &+
    \nu
    \sum_{l=0}^{k-1}
    \left(
    c^k_{l,\varphi}\partial_{zz} Z^l_3 Z^{\tilde{\alpha}}_\hh u  
    +
    (c^k_{l,\varphi})'\partial_{z} Z^l_3 Z^{\tilde{\alpha}}_\hh u 
    \right)
    \label{EQ51}
    .
   \end{split}
  \end{align}
  We first consider the term
  \begin{align}
  	2\nu c^k_{k-1,\varphi}
  	\partial_{zz} Z^{k-1}_3 Z^{\tilde{\alpha}}_\hh u
  	 =2\nu \frac{c^k_{k-1,\varphi}}{\varphi}\partial_z Z^\alpha u
  	  -2\nu \frac{c^k_{k-1,\varphi}\varphi'}{\varphi} \partial_z Z^{k-1}_3 Z^{\tilde{\alpha}}_\hh u,
  	\llabel{EQ52}
  \end{align}
   which corresponds to
  $(j,l)=(k-1,j)$ in the first sum and 
  $l=k-1$ in the second sum.
  Multiplying by $Z^\alpha u |Z^\alpha u|^{p-2}$ yields
  \begin{align}
   \begin{split}
    \frac{4\nu}{p}&
    \int_{\Omega} \frac{c^k_{k-1,\varphi}}{\varphi}|Z^\alpha u|^\frac{p}{2} \partial_z |Z^\alpha u|^\frac{p}{2}
    \,dx
    -2\nu
    \int_{\Omega} c^k_{k-1,\varphi}\varphi' |\partial_z Z^{k-1}_3 Z^{\tilde{\alpha}}_\hh u|^2
    |Z^\alpha u|^{p-2}
    \,dx
    \\
    &\le
    \epsilon \nu \Vert \partial_z |Z^\alpha u|^{\frac{p}{2}}\Vert_{L^{2}}^2
     +C_{\epsilon} \Vert Z^\alpha u\Vert_{L^{p}}^p
      +C\nu^\frac{p}{2} \Vert \partial_z Z^{k-1}_3 Z^{\tilde{\alpha}}_\hh u\Vert_{L^{p}}^p
     ,
    \label{EQ53}
   \end{split}
  \end{align}
  for $\epsilon>0$ sufficiently small and $C_\epsilon>0$. 
  We may replicate \eqref{EQ53} for the terms in \eqref{EQ51} and obtain
  \begin{align}
   \int_\Omega \left(\nu Z^\alpha \Delta u
   - \nu \Delta Z^\alpha u\right) Z^\alpha u|Z^\alpha u|^{p-2} \,dx
   \le 
   \epsilon \nu \Vert \partial_z |Z^\alpha u|^{\frac{p}{2}}\Vert_{L^{2}}^2
   +C_{\epsilon} \Vert u\Vert_{3,p}^p
   +C\nu^\frac{p}{2} \Vert \partial_z u\Vert_{2,p}^p
   .
   \label{EQ61}
  \end{align}
  Therefore, collecting \eqref{EQ46}, \eqref{conork}, \eqref{EQ56}, \eqref{EQ58}, \eqref{EQ47},
  and \eqref{EQ61},
  absorbing all the factors of 
  $\epsilon \nu \Vert \partial_z |Z^\alpha u|^{\frac{p}{2}}\Vert_{L^{2}}^2$, and integrating in time,
  we obtain 
  \begin{align}
  	\begin{split}
  		\Vert  Z^\alpha u(t)&\Vert_{L^p}^{p}
  		+c_0\nu \int_0^t \int_{\Omega} \left(|\nabla Z^\alpha u|^2 |Z^\alpha u|^{p-2}
  		+ |\nabla |Z^\alpha u|^\frac{p}{2}|^2\right) \,dxds
  		\\&\lec
  		\Vert u_0\Vert_{3,p}^p
  		 +\int_0^t\left(
  		 \Vert u\Vert_{3,p}^{p}
  		 (\Vert u\Vert_{2,\infty}+\Vert u\Vert_{W^{1,\infty}})
  		 +\Vert u\Vert_{3,p}^{p-1}\Vert \nabla p\Vert_{3,p}\right)\,ds
  		 +\nu^{\frac{p}{2}}\int_0^t \Vert \nabla u\Vert_{2,p}^p\,ds
  		 ,
  		\llabel{finalconormal}
  	\end{split}
  \end{align}
  for $t \in [0,T]$. Now, we sum over $\alpha$ using \eqref{base2} and conclude~\eqref{EQ.Con}.
 \end{proof}
 
 \startnewsection{Normal Derivative Estimates}{sec.no}
 
 In this section, we present the conormal derivative bounds for $\nabla u$.
 We note that these estimates are only required when $\nu >0$,
 i.e., for \eqref{ap1} and~\eqref{ap2}.
 
 Rather than analyzing the evolution of $\partial_z u$, we define
 \begin{align}
  \eta = \omega_\hh - 2\mu u_\hh^\perp,
  \label{eta}
 \end{align}
 where $\omega = \curl u$ and 
 $u_\hh^\perp = (-u_2,u_1)^T$. 
 As in~\cite{AK1, MR1}, $\eta$ solves
 \begin{align}
  \begin{split}  
   \eta_t
   -\nu \Delta \eta
   +u\cdot \nabla \eta
   &=
   \omega\cdot \nabla u_\hh
   +2\mu \nabla_\hh^\perp p
  \inin{\Omega}
   \\
   \eta &= 0
   \comma z=0
   ,
   \llabel{EQ.eta}   
  \end{split}
 \end{align}
 where $\nabla_\hh^\perp p = (-\partial_2,\partial_1)^T$.
 Unlike $\omega$ and $\partial_z u$, we have that $\eta$ vanishes on the boundary, and
 \begin{align}
  \Vert \partial_z u\Vert_{m,p}
  \lec
  \Vert \eta\Vert_{m,p}+\Vert u\Vert_{m+1,p}
  \comma
  \Vert \partial_z u\Vert_{L^\infty}
  \lec
  \Vert \eta\Vert_{L^\infty}+\Vert u\Vert_{1,\infty},
  \label{EQ62}
 \end{align}
 implying that $\partial_z u$
 is bounded by the conormal derivatives of $u$ and~$\eta$.
 As a further consequence, we also have
 \begin{align}
 	\Vert \omega_\hh\Vert_{L^\infty}\lec \Vert \eta\Vert_{L^\infty}+ \Vert u\Vert_{L^\infty}
 	\text{ and }
 	\Vert \omega_\hh\Vert_{m,p}\lec \Vert \eta\Vert_{m,p} + \Vert u\Vert_{m,p}
 	\comma m \in \mathbb{N}_0\llabel{EQ07}
 	,
 \end{align}
 and
 \begin{align}
 	\Vert \omega_3\Vert_{L^\infty} \lec \Vert u\Vert_{1,\infty}
 	\text{ and }
 	\Vert \omega_3\Vert_{m,p} \lec \Vert u\Vert_{m+1,p}  
 	\comma m \in \mathbb{N}_0,
 	\label{EQ08}
 \end{align}
 since $\omega_3 = \partial_1 u_2 - \partial_2 u_1$. 
 
 The following proposition is the main result of this section.
 
 \cole
 \begin{Proposition}
  \label{P.Nor}
  Let $\mu \in \mathbb{R}$, $\nu \in (0,\bar{\nu}]$, $p \in (2,\infty)$, and 
  assume that $u$ is a smooth solution of \eqref{NSE0}--\eqref{hnavierbdry} on~$[0,T]$ with a smooth initial datum $u_0$.
  Then we have the inequality
  \begin{align}
   \begin{split}
    \Vert \eta(t)&\Vert_{1,p}^2
    +
    c_0 \nu \sum_{0\le |\alpha| \le 1}
    \left(
    \int_0^t \int_\Omega
    \left(
    |\nabla Z^\alpha \eta|^2 |Z^\alpha\eta|^{p-2} + |\nabla |Z^\alpha\eta|^{\frac{p}{2}}|^2 
    \right) \,dxds\right)
    \\&
    \lec
    \Vert \eta_0\Vert_{1,p}^p
    + \int_0^t \left(
    \Vert \eta\Vert_{1,p}^p
    (\Vert u\Vert_{2,\infty}+\Vert \eta\Vert_{L^\infty}+1)
    +\Vert \eta\Vert_{1,p}^{p-1}
    \Vert u\Vert_{2,p}(\Vert \eta\Vert_{L^\infty}
    +\Vert u\Vert_{2,\infty})\right)
    \,ds
    \\&\indeq
    +\int_0^t \Vert \eta\Vert_{1,p}^{p-1}\Vert p\Vert_{2,p} \,ds
    +\nu^{\frac{p}{2}}\int_0^t \Vert \partial_z \eta\Vert_{L^{p}}^p \,ds
    ,\label{EQ.Nor}
   \end{split}
  \end{align}
where $c_0>0$ and $t \in [0,T]$. Moreover,
for $|\alpha| =2$, we have 
  \begin{align}
	\begin{split}
		&\Vert Z^\alpha \eta(t)\Vert_{L^p}^2
		+
		c_0 \nu \sum_{\alpha}
		 \left(
		  \int_0^t \int_\Omega
		   \left(
		    |\nabla Z^\alpha \eta|^2 |Z^\alpha\eta|^{p-2} + |\nabla |Z^\alpha\eta|^{\frac{p}{2}}|^2 
		     \right) \,dxds\right)
		\\&
		\lec
		\Vert Z^\alpha\eta_0\Vert_{2,p}^p
		+ \int_0^t \left(
		\Vert \eta\Vert_{2,p}^p
		(\Vert u\Vert_{2,\infty}+\Vert \eta\Vert_{L^\infty}+1)
		+\Vert \eta\Vert_{2,p}^{p-1}
		\Vert u\Vert_{3,2p}(\Vert \eta\Vert_{L^\infty}
		+\Vert u\Vert_{2,\infty})
		\right)
		\,ds
		\\&\indeq
		+\int_0^t \Vert \eta\Vert_{2,p}^{p-1}\Vert p\Vert_{3,p} \,ds
		+\nu^{\frac{p}{2}}\int_0^t \Vert \partial_z \eta\Vert_{1,p}^p \,ds
		,\label{EQ.Nor2}
	\end{split}
\end{align}
 
\end{Proposition}
\colb
 
 \begin{proof}[Proof of Proposition~\ref{P.Nor}]
  
  We start with $L^p$ estimates obtaining
  \begin{align}
  	\begin{split}
  		\Vert \eta(t)&\Vert_{L^p}^p
  		+\nu \int_0^t \int_\Omega
  		\left(|\nabla \eta|^2 |\eta|^{p-2} + |\nabla |\eta|^{\frac{p}{2}}|^2 \right) \,dxds
  		\\&
  		\lec
  		\Vert \eta_0\Vert_{L^p}^p
  		+\int_0^t 
  		\left(\Vert \eta\Vert_{L^p}^p
  		(\Vert \eta\Vert_{L^\infty}+\Vert u\Vert_{1,\infty}) 
  		+\Vert \eta\Vert_{L^p}^{p-1} 
  		(\Vert u\Vert_{1,p}\Vert \eta\Vert_{L^\infty}
  		+ \Vert u\Vert_{1,p}\Vert u\Vert_{1,\infty}
  		+\Vert p\Vert_{1,p})
  		\right)\,ds
  		.
  		\label{EQ09}
  	\end{split}
  \end{align}
  where we have used \eqref{EQ62}--\eqref{EQ08} to write
  \begin{align}
  	\Vert \omega \cdot \nabla u_\hh\Vert_{L^p}
  	 \lec
  	  \Vert \omega\Vert_{L^{\infty}}\Vert \nabla u_\hh\Vert_{L^{p}}
  	   \lec
  	    (\Vert \eta\Vert_{L^{\infty}}+\Vert u\Vert_{1,\infty})(\Vert \eta\Vert_{L^p}+\Vert u\Vert_{1,p})
  .\label{EQM09}
  \end{align}
  Now, for $\alpha \in \mathbb{N}^3_0$, $Z^\alpha \eta$ solves
  \begin{align}
  	\begin{split}
   (\partial_t -\nu\Delta+u\cdot \nabla )Z^\alpha \eta
   &=
   Z^\alpha (\omega\cdot \nabla u_\hh)
   +2\mu Z^\alpha \nabla_\hh^\perp p
   +(u\cdot \nabla Z^\alpha \eta
   -Z^\alpha (u \cdot \nabla \eta))
   +\nu (Z^\alpha \Delta \eta
   -\Delta Z^\alpha \eta)
   \\
   &=\sum_{i=1}^{4} \mathcal{R}_{\alpha,i}
   ,
   \label{EQ63}
   \end{split}
  \end{align}
  with $Z^\alpha \eta |_{\partial \Omega} = 0$.
  Upon testing with $Z^\alpha \eta |Z^\alpha \eta|^{p-2}$,
  we obtain
  \begin{align}
  	\frac{1}{p}\frac{d}{dt}\Vert Z^\alpha\eta\Vert_{L^p}^p
  	+\nu \int_\Omega
  	\left(|\nabla Z^\alpha\eta|^2 |Z^\alpha\eta|^{p-2} 
  	+ 4\frac{p-2}{p^2}|\nabla |Z^\alpha\eta|^{\frac{p}{2}}|^2 \right) \,dx
  	=\sum_{i=1}^{4} (\mathcal{R}_{\alpha,i}, Z^\alpha \eta |Z^\alpha \eta|^{p-2})
  	.
  	\llabel{EQM009}
  \end{align}
  First, we consider a single horizontal derivative, i.e., 
  $\alpha = (\alpha_\hh,0)$
  where $|\alpha_\hh| =1$. 
  When this is the case, 
  we have $\mathcal{R}_{\alpha,4} = 0$ and 
  \begin{align}
  	\sum_{i=1}^{3} 
  	 (\mathcal{R}_{\alpha,i}, Z^\alpha \eta |Z^\alpha \eta|^{p-2})
  	 \lec
  	  \Vert \eta\Vert_{1,p}^{p-1}
  	   (\Vert Z^\alpha (\omega \cdot \nabla u_\hh)\Vert_{L^{p}}
  	    +\Vert p\Vert_{2,p}+\Vert u\cdot \nabla Z^\alpha \eta
  	    -Z^\alpha (u \cdot \nabla \eta)\Vert_{L^{p}})
  	  .\label{EQM10}
  \end{align}
  Expanding $Z^\alpha (\omega \cdot \nabla u)$, we get
  \begin{align}
  	\begin{split}
  	\Vert Z^\alpha (\omega \cdot \nabla u_\hh)\Vert_{L^{p}}
  	 &\lec
  	  \Vert Z^\alpha \omega_\hh \cdot \nabla_\hh u_\hh \Vert_{L^{p}}
  	   +\Vert Z^\alpha \omega_3  \partial_z u_\hh\Vert_{L^{p}}
  	    +\Vert \omega_\hh \cdot Z^\alpha \nabla_\hh u_\hh \Vert_{L^{p}}
  	     +\Vert \omega_3  Z^\alpha \partial_z u_\hh\Vert_{L^{p}}
  	  \\&\lec
  	   \Vert \eta\Vert_{1,p}\Vert u\Vert_{2,\infty}+
  	   \Vert \eta\Vert_{L^{\infty}}\Vert u\Vert_{2,p}+
  	   \Vert u\Vert_{2,p}\Vert u\Vert_{2,\infty}
  	     ,\label{EQM11}
  	     \end{split}
  \end{align}
  where we have used \eqref{EQ62}--\eqref{EQ08}.
  Next, we rewrite the commutator term on the right-hand side of \eqref{EQM10} 
  and employ \eqref{EQ.u3} to obtain
  \begin{align}
  	\Vert u\cdot \nabla Z^\alpha \eta-Z^\alpha (u \cdot \nabla \eta)\Vert_{L^{p}}
  	 \lec
  	  \Vert Z^\alpha u_\hh \cdot \nabla_\hh \eta\Vert_{L^{p}}
  	   +\left\Vert Z^\alpha \frac{u_3}{\varphi} \cdot Z_3 \eta \right\Vert_{L^{p}}
  	    \lec
  	     \Vert u\Vert_{2,\infty}\Vert \eta\Vert_{1,p}
  	     ,\label{EQM12}
  \end{align}
  recalling that $|\alpha| = 1$ and $Z_\hh$ commutes with $\partial_z$. 
  Now,
  we consider $Z^\alpha = Z^3$ and 
  estimate $\mathcal{R}_{\alpha,i}$, $i=1,2,3$
  by repeating \eqref{EQM10}--\eqref{EQM12}.
  We note in passing that when we redo \eqref{EQM12},
  we commute $\partial_z$ and $Z_3$ using Lemma~\ref{L01} and
  obtain $C\Vert u \eta\Vert_{L^{p}}$ which is a lower order term.
  Therefore, it only remains to estimate $\mathcal{R}_{\alpha,4}$,
  which we expand as
  \begin{align}
  	\mathcal{R}_{\alpha,4}
  	 =-\nu 
  	 (2\varphi' \partial_{zz} \eta  + \varphi'' \partial_z \eta)
  	 =-2\nu
  	 \frac{\varphi'}{\varphi}\partial_z Z_3 \eta 
  	  +\nu \left(\frac{\varphi'}{\varphi}-\varphi''\right)\partial_z \eta
  	  .\llabel{EQM13}
  \end{align}
  Recalling that $Z^\alpha = Z_3$, we obtain
  \begin{align}
  	\begin{split}
  	(\mathcal{R}_{\alpha,4},Z^\alpha \eta |Z^\alpha \eta|^{p-2})
  	&=-\frac{4\nu}{p} \int_\Omega 
  	  \frac{\varphi'}{\varphi} |Z^\alpha \eta|^{\frac{p}{2}} \partial_z |Z^\alpha \eta|^{\frac{p}{2}} \,dx
  	 +\nu \int_\Omega 
  	        (\varphi'-\varphi\varphi'')|\partial_z \eta|^2 |Z^\alpha \eta|^{p-2}\,dx
  	    \\&\le 
  	    \epsilon \nu \Vert \partial_z |Z^\alpha \eta|^{\frac{p}{2}}\Vert_{L^{2}}^2
  	    +C_{\epsilon} \Vert Z^\alpha \eta\Vert_{L^{p}}^p
  	    +C\nu^\frac{p}{2} \Vert \partial_z \eta\Vert_{L^{p}}^p
   ,\label{EQM14}  
\end{split}
\end{align}
for $\epsilon>0$ sufficiently small and $C_\epsilon>0$.
  Now, we collect \eqref{EQM09}--\eqref{EQM14}, 
  absorb $\epsilon \nu \Vert \partial_z |Z^\alpha \eta|^{\frac{p}{2}}\Vert_{L^{2}}^2$,
  integrate in time and combine with \eqref{EQ09}, concluding the proof of \eqref{EQ.Nor}.
  
   We proceed to establishing \eqref{EQ.Nor2}, and
   first, consider $\alpha = (\alpha_\hh,0)$ with $|\alpha|=2$.
   Similarly to \eqref{EQM10} and \eqref{EQM11}, we obtain
   \begin{align}
   	\sum_{i=1}^{3} 
   	(\mathcal{R}_{\alpha,i}, Z^\alpha \eta |Z^\alpha \eta|^{p-2})
   	\lec
   	\Vert \eta\Vert_{2,p}^{p-1}
   	(\Vert Z^\alpha (\omega \cdot \nabla u_\hh)\Vert_{L^{p}}
   	+\Vert p\Vert_{3,p}+\Vert u\cdot \nabla Z^\alpha \eta
   	-Z^\alpha (u \cdot \nabla \eta)\Vert_{L^{p}})
   	,\llabel{EQM15}
   \end{align}
   and
   \begin{align}
   	\begin{split}
   		\Vert Z^\alpha (\omega \cdot \nabla u_\hh)\Vert_{L^{p}}
   		&\lec
   		\Vert Z^\alpha \omega \cdot \nabla u_\hh \Vert_{L^{p}}
   		+\Vert Z^\beta \omega_\hh \cdot  Z^{\alpha-\beta} \nabla_\hh u_\hh\Vert_{L^{p}}
   		+\Vert Z^\beta \omega_3  Z^{\alpha-\beta} \partial_z u_\hh \Vert_{L^{p}}
   		+\Vert \omega \cdot  Z^\alpha \nabla u_\hh\Vert_{L^{p}}
   		\\&\lec
   		(\Vert \eta\Vert_{2,p}+\Vert u\Vert_{3,p})(\Vert \eta\Vert_{L^{\infty}}+\Vert u\Vert_{2,\infty})
   		,\llabel{EQM16}
   	\end{split}
   \end{align}
   where $|\beta|=1$. We note that
   \begin{align}
   	   \Vert Z^\beta \omega_3  Z^{\alpha-\beta} \partial_z u_\hh \Vert_{L^{p}}
   	    \lec
   	     \Vert \omega_3\Vert_{L^{\infty}}\Vert \partial_z u_\hh\Vert_{2,p}
   	      +\Vert \omega_3\Vert_{2,p}\Vert \partial_z u_\hh\Vert_{L^\infty}
   	       \lec
   	        (\Vert \eta\Vert_{2,p}+\Vert u\Vert_{3,p})(\Vert \eta\Vert_{L^{\infty}}+\Vert u\Vert_{2,\infty})
   	      ,\llabel{EQM17}
   \end{align}
   by Lemma~\ref{L03} and \eqref{EQ62}--\eqref{EQ08}.
   Next, we estimate the commutator term $\mathcal{R}_{\alpha,3}$ as
   \begin{align}
   	\begin{split}
   	\Vert u\cdot \nabla Z^\alpha \eta-Z^\alpha (u \cdot \nabla \eta)\Vert_{L^{p}}
   	 &\lec
   	  \Vert Z^\alpha u_3 \partial_z \eta\Vert_{L^{p}}
   	   +\Vert Z^\beta u_3 Z^{\alpha -\beta}\partial_z \eta\Vert_{L^{p}}
   	    +\Vert Z^\alpha u_\hh \cdot \nabla_\hh \eta\Vert_{L^{p}}
   	     +\Vert Z^\beta u_\hh Z^{\alpha-\beta} \nabla_\hh \eta\Vert_{L^{p}}
   	  \\&   = I_1+I_2+I_3+I_4
   	     \comma |\beta|=1
   	     ,\llabel{EQM18}
   	     \end{split}
   \end{align}
   For $I_3$ and $I_4$, we have
   \begin{align}
   	 I_3+I_4
   	  \lec
   	   \Vert u\Vert_{2,\infty}\Vert \eta\Vert_{2,p}
   	    ,\llabel{EQM19}
   \end{align}
   while for $I_2$ we conormalize and write
   \begin{align}
   	I_2
   	 =
   	  \left\Vert Z^\beta\frac{u_3}{\varphi} Z^{\alpha-\beta} Z_3 \eta\right\Vert_{L^{p}}
   	   \lec
   	    \Vert u\Vert_{2,\infty}\Vert \eta\Vert_{2,p}
   	    .\llabel{EQM20}
   \end{align}
   Finally, for $I_3$, we employ \eqref{EQgag} to obtain
   \begin{align}
   	I_1
   	 \lec
   	  \left\Vert Z^\alpha \frac{u_3}{\varphi}\right\Vert_{L^{2p}}
   	   \Vert Z_3 \eta\Vert_{L^{2p}}
   	    \lec
   	     \Vert u\Vert_{3,2p}\Vert \eta\Vert_{L^\infty}^\frac12\Vert \eta\Vert_{2,p}^\frac12
   	     ,\llabel{EQM21}
   \end{align}
   concluding the estimate on $\mathcal{R}_{\alpha,3}$.
   
   We still need to consider $Z^\alpha$ where $\alpha = (\alpha_\hh,k)$
   and $k=1,2$. Recalling \eqref{EQ63}, we note that $\mathcal{R}_{\alpha,i}$,
   $i=1,2,3$, are treated as in the previous paragraph upon employing
   Lemma~\ref{L01} and introducing low-order terms. Therefore, we only 
   analyze $\mathcal{R}_{\alpha,4}$. We expand this term as
   \begin{equation}
   	\llabel{EQ105}
   	\nu Z^\alpha \Delta \eta
   	- \nu \Delta Z^\alpha \eta =
   	\begin{cases}
   		-\nu 
   		(2\varphi' \partial_{zz} Z_\hh \eta  + \varphi'' \partial_z Z_\hh \eta), &k=1  \\
   		-\nu 
   		(4\varphi' \partial_{zz} Z_3 \eta
   		-4 (\varphi')^2 \partial_{zz} \eta
   		+4\varphi'' \partial_z Z_3 \eta
   		-5 \varphi' \varphi'' \partial_z \eta
   		+\varphi''' Z_3 \eta), &k=2
   		,
   	\end{cases}
   \end{equation}
   and write
    \begin{align}
   	2\nu c^k_{k-1,\varphi}
   	\partial_{zz} Z^{k-1}_3 Z^{\tilde{\alpha}}_\hh \eta
   	=2\nu \frac{c^k_{k-1,\varphi}}{\varphi}\partial_z Z^\alpha \eta
   	-2\nu \frac{c^k_{k-1,\varphi}\varphi'}{\varphi} \partial_z Z^{k-1}_3 Z^{\tilde{\alpha}}_\hh \eta,
   	\llabel{EQ52}
   \end{align}
   from where we obtain that
   \begin{align}
   		(\mathcal{R}_{\alpha,4},Z^\alpha \eta |Z^\alpha \eta|^{p-2})
   		\le 
   		\epsilon \nu \Vert \partial_z |Z^\alpha \eta|^{\frac{p}{2}}\Vert_{L^{2}}^2
   		+C_{\epsilon} \Vert Z^\alpha \eta\Vert_{2,p}^p
   		+C\nu^\frac{p}{2} \Vert \partial_z \eta\Vert_{1,p}^p
   		,\llabel{EQM23}  
   	\end{align}
   for $\epsilon>0$ sufficiently small and $C_\epsilon>0$.
   Upon collecting the estimates for $\mathcal{R}_{\alpha,i}$,
   absorbing $\epsilon \nu \Vert \partial_z |Z^\alpha \eta|^{\frac{p}{2}}\Vert_{L^{2}}^2$,
   and integrating in time, we conclude \eqref{EQ.Nor2}.
 \end{proof}
 
 \startnewsection{Pressure Estimates}{sec.p}
 
 In this section, we establish bounds on the pressure term using
 the normal and conormal derivatives of $u$. 
 To achieve this, we use \eqref{NSE0} and \eqref{hnavierbdry} and 
 note that $p$ solves the elliptic
 Neumann problem
 \begin{align}
 	\begin{split}
 		-\Delta p &= \partial_i u_j \partial_j u_i,
 		\hspace{0.7cm } \inon{ in $\Omega \times (0,T)$},\\
 		\nabla p \cdot n &= -2\mu \nu \nabla_\hh \cdot u_\hh,
 		\inon{ on $\partial \Omega \times (0,T)$}
 		.\label{pre}
 	\end{split} 
 \end{align}
 Applying $Z^\alpha = Z_\hh^{\tilde{\alpha}} Z_3^k$ to this system, we obtain
 \begin{align}
 	\begin{split}
 		-\Delta Z^\alpha p  
 		&= Z^\alpha(\partial_i u_j \partial_j u_i)
 		+ Z^\alpha \Delta p -\Delta Z^\alpha p,
 		\inon{ in $\Omega \times (0,T)$},\\
 		\nabla Z^\alpha p \cdot n
 		&=-2\mu \nu \tilde{c}_{0,\varphi}^{k+1} Z^{\tilde{\alpha}}_\hh \nabla_\hh \cdot u_\hh,
 		\hspace{1.5 cm}
 		\inon{ on $\partial \Omega \times (0,T)$ }   
 		,\label{kpre}
 	\end{split}  
 \end{align}
 where we have used $Z_3 = 0$ on $\partial \Omega$ and
 \begin{align}
 	\nabla Z^\alpha p \cdot n
 	= -\partial_z Z^{\tilde{\alpha}}_\hh Z^{k+1}_3 p
 	= -Z^{\tilde{\alpha}}_\hh \sum_{j=0}^{k}
 	\tilde{c}_{j,\varphi}^{k+1} Z_3^j \partial_z p 
 	= -\tilde{c}_{0,\varphi}^{k+1} Z^{\tilde{\alpha}}_\hh \partial_z p 
 	= -2\mu \nu \tilde{c}_{0,\varphi}^{k+1} Z^{\tilde{\alpha}}_\hh \nabla_\hh \cdot u_\hh    
 	.\llabel{kpre.bdry}
 \end{align}
 In the following proposition, we estimate $Z^\alpha p$ for $|\alpha|\le 2$.
 
 \cole
 \begin{Proposition}
 	\label{P.Pre}
 	Let $\mu \in \mathbb{R}$, $\nu\in [0,\bar{\nu}]$, $p \in (2,\infty)$, and assume that $(u,p)$ is a smooth solution of \eqref{NSE0}--\eqref{hnavierbdry} on $[0,T]$.
 	Then we have the inequality
 	\begin{align}
 		\Vert D^2 p(t)\Vert_{j,p}+\Vert \nabla p(t)\Vert_{j,p}
 		\lec 
 		\Vert u(t)\Vert_{j+1,p}(\Vert u(t)\Vert_{2,\infty}+\Vert \eta(t)\Vert_{L^\infty}+1)
 		+\nu \sum_{1\le |\tilde{\theta}|\le j+1}\Vert u(t)\Vert_{3,p}^{\frac{p-1}{p}}
 		\Vert \nabla Z^{\tilde{\theta}}_\hh u(t)\Vert_{L^p}^{\frac{1}{p}}
 		,
 		\label{EQ.Pre}
 	\end{align}
 	for $j=0,1,2$ and $t \in [0,T]$. 
 \end{Proposition}
 \colb
 
 First, we estimate $ Z^{\tilde{\alpha}}_\hh p$ for $|\tilde{\alpha}|\le 2$, and then
 we inductively estimate $Z_\hh^{\tilde{\alpha}} Z_3^k$ 
 for $0\le k \le 2$ and $|\tilde{\alpha}|+k \le 2$.
 
 \cole
 \begin{Lemma}
 	\label{L.Pre}
 	Under the assumption of Proposition~\ref{P.Pre}, we have the inequality
 	\begin{align}
 		\Vert D^2 Z^{\tilde{\alpha}}_\hh p\Vert_{L^p}+
 		\Vert \nabla Z^{\tilde{\alpha}}_\hh p\Vert_{L^p}
 		\lec 
 		\Vert u(t)\Vert_{3,p}(\Vert u(t)\Vert_{2,\infty}+\Vert \eta(t)\Vert_{L^\infty})
 		+\nu \sum_{1\le |\tilde{\theta}|\le 3}\Vert u(t)\Vert_{3,p}^{\frac{p-1}{p}}
 		\Vert \nabla Z^{\tilde{\theta}}_\hh u(t)\Vert_{L^p}^{\frac{1}{p}},
 		\label{EQ.hPre}
 	\end{align}
 	where $0 \le |\tilde{\alpha}|\le 2$.
 \end{Lemma}
 \colb
 
 \begin{proof}[Proof of Lemma~\ref{L.Pre}]
 	We consider \eqref{kpre} with $Z^\alpha = Z^{\tilde{\alpha}}$ and $0\le |\tilde{\alpha}|\le 2$. 
 	Employing the $W^{2,p}$ elliptic estimate for the Neumann problem and the trace theorem, it follows that
 	\begin{align}
 		\Vert D^2 Z^{\tilde{\alpha}}_\hh p\Vert_{L^2}
 		+
 		\Vert \nabla Z^{\tilde{\alpha}}_\hh p\Vert_{L^2}
 		\lec
 		\Vert Z^{\tilde{\alpha}}_\hh (\partial_i u_j \partial_j u_i)\Vert_{L^p}
 		+
 		2|\mu| \nu \Vert u\Vert_{3,p}^{\frac{p-1}{p}}
 		\Vert \nabla Z^{\tilde{\alpha}}_\hh \nabla_\hh u\Vert_{L^p}^{\frac{1}{p}}
 		.\label{EQ87}
 	\end{align}
 	To bound the quadratic term on \eqref{EQ87}, 
 	we only consider  
 	$Z^{\tilde{\alpha}}_\hh (Z_\hh u_\hh Z_\hh u_\hh )$ 
 	and 
 	$Z^{\tilde{\alpha}}_\hh (Z_\hh u_3 \partial_z u_\hh)$  
 	since 
 	\begin{align}
 		(\partial_z u_3)^2
 		=(\nabla_\hh \cdot u_\hh)^2
 		.\llabel{EQM24}
 	\end{align}
 	Now, utilizing \eqref{EQ.int}, we get
 	\begin{align}
 		\Vert Z^{\tilde{\alpha}}_\hh (Z_\hh u_\hh Z_\hh u_\hh )\Vert_{L^p}
 		\lec
 		\Vert u\Vert_{3,p}\Vert u\Vert_{1,\infty},
 		\label{EQ89}
 	\end{align}
 	while for the term involving $\partial_z u$,
 	we conormalize when necessary and obtain
 	\begin{equation}
 		\label{EQ91b}
 		\Vert Z_\hh^{\tilde{\beta}} Z_\hh u_3 
 		\partial_z Z_\hh^{\tilde{\alpha} - {\tilde{\beta}}} u_\hh\Vert_{L^p}
 		\lec
 		\begin{cases}
 			\bigl\Vert Z_\hh \frac{u_3}{\varphi}\bigr\Vert_{L^\infty}\Vert Z^{\tilde{\alpha}-\tilde{\beta}}_\hh Z_3 u\Vert_{L^p}
 			\lec \Vert u\Vert_{2,\infty}\Vert u\Vert_{3,p}, & |{\tilde{\beta}}|=0
 			\\
 			\bigl\Vert Z_\hh^{\tilde{\beta}} Z_\hh \frac{u_3}{\varphi}\bigr\Vert_{L^p}\Vert Z_\hh Z_3 u\Vert_{L^\infty}
 			\lec\Vert u\Vert_{3,p}\Vert u\Vert_{2,\infty}, & |{\tilde{\beta}}|=1
 			\\
 			\Vert u\Vert_{3,p}\Vert \partial_z u\Vert_{L^\infty}, & |{\tilde{\beta}}|=2,
 		\end{cases}
 	\end{equation}
 	where we have used \eqref{EQ.u3}. 
 	Recalling \eqref{EQ62}, we conclude \eqref{EQ.hPre}.
 \end{proof}
 
 Now, we prove Proposition~\ref{P.Pre}.
 
 \begin{proof}[Proof of Proposition~\ref{P.Pre}]
 	
 	We only give details for the terms
 	$\Vert D^2 Z^\alpha p\Vert_{L^p}$ and
 	$\Vert \nabla Z^\alpha p\Vert_{L^p}$ with $|\alpha|=2$,
 	utilizing induction on $0\le k \le 3$ such that 
 	$k + |\tilde{\alpha}| = 2$. The base step, i.e., $k=0$ 
 	is a consequence of Lemma~\ref{L.Pre}.
 	Now, the induction assumption is 
 	\begin{align}
 		\begin{split}
 			\Vert D^2 Z^{\tilde{\beta}}_\hh Z^{k}_3 p\Vert_{L^p}+
 			\Vert \nabla Z^{\tilde{\beta}}_\hh Z^{k}_3 p\Vert_{L^p}
 			&\lec \Vert u\Vert_{3}(\Vert u\Vert_{2,\infty}+\Vert \eta\Vert_{L^\infty}+1) 
 			\\&\indeq+
 			\nu \sum_{1\le |\tilde{\theta}|\le |\tilde{\beta}|+1}\Vert u(t)\Vert_{3,p}^{\frac{p-1}{p}}
 			\Vert \nabla Z^{\tilde{\theta}}_\hh u(t)\Vert_{L^p}^{\frac{1}{p}} 
 			\comma |\tilde{\beta}| + k = 2,
 			\label{EQ.pre.k-1} 
 		\end{split}
 	\end{align}
 	and we need to establish
 	\begin{align}
 		\begin{split}
 			\Vert D^2 Z^{\tilde{\alpha}}_\hh Z^{k+1}_3 p\Vert_{L^p}+
 			\Vert \nabla Z^{\tilde{\alpha}}_\hh Z^{k+1}_3 p\Vert_{L^p}
 			&\lec \Vert u\Vert_{3}(\Vert u\Vert_{2,\infty}+\Vert \eta\Vert_{L^\infty}+1) 
 			\\
 			&\indeq+
 			\nu \sum_{1\le |\tilde{\theta}|\le |\tilde{\alpha}|+1}\Vert u(t)\Vert_{3,p}^{\frac{p-1}{p}}
 			\Vert \nabla Z^{\tilde{\theta}}_\hh u(t)\Vert_{L^p}^{\frac{1}{p}} 
 			\comma |\tilde{\alpha}| + k + 1 = 2.
 			\label{EQ.pre.k}
 		\end{split}
 	\end{align}
 	
 	Let $Z^\alpha = Z^{\tilde{\alpha}}_\hh Z^{k+1}_3$ and employ
 	the elliptic estimates for \eqref{kpre} obtaining
 	\begin{align}
 		\Vert D^2 Z^\alpha p\Vert_{L^p}+
 		\Vert \nabla Z^\alpha p\Vert_{L^p}
 		\lec
 		\Vert Z^\alpha(\partial_i u_j \partial_j u_i)\Vert_{L^p}
 		+
 		\Vert Z^\alpha \Delta p -\Delta Z^\alpha p\Vert_{L^p}
 		+
 		\nu \Vert u\Vert_{3,p}^{\frac{p-1}{p}}
 		\Vert \nabla Z^{\tilde{\alpha}}_\hh \nabla_\hh u\Vert_{L^p}^{\frac{1}{p}}
 		,\label{EQ93}
 	\end{align}
 	using that $\mu$ is a constant. 
 	For the quadratic term in \eqref{EQ93} we proceed as in \eqref{EQ89} and \eqref{EQ91b}
 	to obtain
 	\begin{align}
 		\Vert Z^\alpha(\partial_i u_j \partial_j u_i)\Vert_{L^p}
 		\lec 
 		\Vert u\Vert_{3,p}(\Vert u\Vert_{2,\infty}+\Vert \eta\Vert_{L^\infty})
 		,
 		\label{EQ94}
 	\end{align}
 	upon introducing lower-order terms.
 	Now, we rewrite the commutator term for the pressure as 
 	\begin{equation}
 		\label{EQ95}
 		Z^\alpha \Delta p
 		- \Delta Z^\alpha p =
 		\begin{cases}
 			- 
 			(2\varphi' Z_\hh\partial_{zz} p  + \varphi''Z_\hh  \partial_z p), &k=1  \\
 			- 
 			(2\varphi' Z_3 \partial_{zz} p
 			+2 (\varphi')^2 \partial_{zz} p
 			+2\varphi'' Z_3 \partial_z p
 			+3 \varphi' \varphi'' \partial_z p
 			+\varphi''' Z_3 p), &k=2
 			.
 		\end{cases}
 	\end{equation}
 	Using \eqref{pre}, we arrive at
 	\begin{align}
 		Z\partial_{zz} p
 		=
 		-Z \Delta_\hh p
 		-Z (\partial_i u_j \partial_j u_i)
 		\comma Z=Z_\hh, Z_3.
 		\label{EQ97}
 	\end{align}
 	We may estimate the quadratic term in \eqref{EQ97} as in \eqref{EQ89} and~\eqref{EQ91b},
 	while for the pressure term we have
 	\begin{align}
 		\Vert Z\Delta_\hh p\Vert_{L^p}
 		\lec
 		\Vert D^2 Z p\Vert_{L^p}
 		\lec 
 		\Vert u\Vert_{3}(\Vert u\Vert_{2,\infty}+\Vert \eta\Vert_{L^\infty}+1)
 		+
 		\nu \sum_{1\le |\tilde{\theta}|\le 3-k}\Vert u(t)\Vert_{3,p}^{\frac{p-1}{p}}
 		\Vert \nabla Z^{\tilde{\theta}}_\hh u(t)\Vert_{L^p}^{\frac{1}{p}} 
 		,\llabel{EQ98}
 	\end{align}
 	employing the induction assumption~\eqref{EQ.pre.k-1}.
 	We may estimate the remaining lower orders terms in \eqref{EQ95}
 	and write
 	\begin{align}
 		\Vert Z^\alpha \Delta p -\Delta Z^\alpha p\Vert_{L^2}
 		\lec \Vert u\Vert_{4}(\Vert u\Vert_{2,\infty}+\Vert \eta\Vert_{L^\infty}+1)
 		+
 		\nu \sum_{1\le |\tilde{\theta}|\le 3}\Vert u(t)\Vert_{3,p}^{\frac{p-1}{p}}
 		\Vert \nabla Z^{\tilde{\theta}}_\hh u(t)\Vert_{L^p}^{\frac{1}{p}} 
 		.\label{EQ99}
 	\end{align}
 	Finally, combining \eqref{EQ93}, \eqref{EQ94}, and \eqref{EQ99},
 	we conclude \eqref{EQ.pre.k} and the proof of Proposition~\ref{P.Pre}.
 \end{proof}
 
 \startnewsection{$L^\infty$ estimates}{sec.inf}

 In this section, we establish a control over $\Vert \eta\Vert_{L^\infty}$,
 $\Vert \omega\Vert_{L^{\infty}}$, and~$\Vert u\Vert_{2,\infty}$.
 
 \cole
 \begin{Proposition}
  \label{P.Inf}
  Let $\mu \in \mathbb{R}$, $\nu\in [0,\bar{\nu}]$, and assume that $(u,p)$ is a smooth solution of \eqref{NSE0}--\eqref{hnavierbdry} on $[0,T]$ with a smooth initial datum $u_0$.
  Then we have the following inequalities.
  \begin{itemize}
   \item[i.] If $\mu \in \mathbb{R}$, we have
   \begin{align}
     \Vert \eta(t)\Vert_{L^\infty}
     \lec 
     \Vert \eta_0\Vert_{L^\infty}
     + \int_0^t \biggl((\Vert u\Vert_{1,\infty}+\Vert \eta\Vert_{L^\infty})^2
     +\Vert Z_\hh p\Vert_{L^\infty}\biggr)\,ds
     \label{EQ.Inf1}
    \end{align}
   for $t \in [0,T]$.
   \item[ii.] If $\mu \ge 0$, we have
   \begin{align}
    \begin{split}
     \Vert u(t)&\Vert_{2,\infty}^2+\Vert \eta(t)\Vert_{L^\infty}^2
     \\&\lec 
     \Vert u_0\Vert_{2,\infty}^2+\Vert \eta_0\Vert_{L^\infty}^2
     + \int_0^t \biggl((\Vert u\Vert_{2,\infty}+\Vert \eta\Vert_{L^\infty}+1)^3
     +\Vert u\Vert_{2,\infty}\Vert \nabla p\Vert_{2,\infty}\biggr)\,ds
     \\&\indeq
     +\nu\int_0^t \Vert \partial_z u\Vert_{1,\infty}^2\,ds
     ,
     \label{EQ.Inf2}
    \end{split}
   \end{align}
   for $t\in [0,T]$.
   \item[iii.] If $\nu = 0$, we have
   \begin{align}
   	\begin{split}
   		\Vert u(t)&\Vert_{2,\infty}^2+\Vert \omega(t)\Vert_{L^\infty}^2
   		\\&\lec 
   		\Vert u_0\Vert_{2,\infty}^2+\Vert \omega_0\Vert_{L^\infty}^2
   		+ \int_0^t \biggl((\Vert u\Vert_{2,\infty}+\Vert \eta\Vert_{L^\infty}+1)^3
   		+\Vert u\Vert_{2,\infty}\Vert \nabla p\Vert_{2,\infty}\biggr)\,ds
   		,
   		\label{EQ.Inf3}
   	\end{split}
   \end{align}
   for $t \in [0,T]$.
  \end{itemize} 
 \end{Proposition}
 \colb
 
 We note that Proposition~\ref{P.Inf}(i) is needed
 for~\eqref{ap2}.
 Therefore, we rely on the $W^{1,3+\delta} \subset L^\infty$ embedding 
 rather than propagate $u$ in $L^\infty(0,T; W^{2,\infty})$. 
 When either $\mu \ge 0$ or $\nu = 0$, we do not
 assume that $\partial_z u$ is twice conormal differentiable.
 Hence, we estimate $\Vert u\Vert_{2,\infty}$ directly
 to establish \eqref{ap2} and \eqref{ap3}. 
 
 We note that the proof of Proposition~\ref{P.Inf}
 is almost identical to the proofs of~\cite[Proposition~6.1]{AK1}
 and~\cite[Proposition~3.3]{AK2}. The main idea is to
 perform $L^q$~estimates and send $q \to \infty$.
 The only difference is that we estimate $\nu \Vert \partial_z u\Vert_{2,\infty}$
 in Section~\ref{sec.max} and $\Vert \nabla p\Vert_{2,\infty}$
 in Section~\ref{sec.apri}.
 
 \startnewsection{Maximal parabolic regularity estimates}{sec.max}
 
 In this section, we estimate the higher order normal and conormal derivatives
 of $\partial_z u$, e.g., $\nu^\frac12 \partial_z u\in L^{p}W^{2,p}_\cco$ or
 $\nu^\frac12 \partial_z \eta \in L^p W^{1,p}_\cco$. Such terms either result from
 trace inequalities or commuting $Z_3$ with $\Delta$.  
 Therefore, we need the maximal regularity properties
 of the heat equation. To present them, let $t \in [0,T]$. Denoting
 $\Omega_t = (0,t)\times \Omega$ and 
 $\partial\Omega_t = (0,t)\times \partial\Omega$,
 we introduce the space-time anisotropic Sobolev
 spaces
 \begin{align}
 	W_{t,x}^{m,p}(\Omega_t)=W^{m,p}(\Omega_t)
 	=&
 	\{f \in L^p(\Omega_t) : \partial^r_t D_x^\alpha f \in L^p(\Omega_t), (r,\alpha) \in \mathbb{N}_0^4, 2r+|\alpha|\le m  \}
 	,
 	\label{ani.def}
 \end{align} 
 with the norms
 \begin{align}
 	\Vert f\Vert_{W^{m,p}(\Omega_t)}=
 	\sum_{2r+|\alpha|\le m} \Vert \partial^r_t D_x^\alpha f\Vert_{L^{p}(\Omega_t)}
 	,\llabel{norm.ani} 
 \end{align}
 for $m \in \mathbb{N}_0$. When $s>0$ and non-integer, let 
 $W^{s,p}(\Omega_t)$ be the space of functions
 with a finite norm
 \begin{align}
 	\begin{split}
 		\Vert f\Vert_{W^{s,p}(\Omega_t)}=&
 		\sum_{2r+|\alpha|<s}
 		\Vert \partial_t^r D_x^\alpha f\Vert_{L^{p}(\Omega_t)}
 		+
 		\sum_{2r+|\alpha|=\lfloor s\rfloor}
 		[\partial_t^r D_x^\alpha f]_{s-\lfloor s \rfloor,x,p,\Omega_t}
 		\\&+
 		\sum_{0<s-2r-|\alpha|<2}
 		[\partial_t^r D_x^\alpha f]_{\frac{s-2r-|\alpha|}{2},t,p,\Omega_t}
 		,\llabel{nonintnorm}
 	\end{split}
 \end{align}
 where the seminorms $[f]_{s',x,p,\Omega_t}$
 and $[f]_{s',t,p,\Omega_t}$ are given by
 \begin{align}
 	[f]_{s',x,p,\Omega_t}
 	=
 	\left(\int_{0}^t \int_{\Omega\times \Omega} \frac{|f(x,t)-f(y,t)|^p}{|x-y|^{3+ps'}}\,dxdydt\right)^\frac{1}{p}
 	,\label{spacesemi}
 \end{align}
 and
 \begin{align}
 	[f]_{s',t,p,\Omega_t}
 	=
 	\left(\int_{(0,t)^2}\int_{\Omega} \frac{|f(x,t)-f(x,\tau)|^p}{|t-\tau|^{1+ps'}}\,dxdydt\right)^\frac{1}{p}
 	.\label{timesemi}
 \end{align}
 Finally, we note that the spaces
 $W^{k,p}(\partial\Omega_t)$ and 
 $W^{s,p}(\partial\Omega_t)$ are defined
 upon replacing $\Omega_t$ with $\partial\Omega_t$ in \eqref{ani.def}--\eqref{timesemi}
 and $3+ps'$ with $2+ps'$ in \eqref{spacesemi}.
 Now, we state certain maximal regularity estimates
 for the heat equation.
 
 \cole
 \begin{Lemma}
 	\label{Lh}
 	Let $T,\nu>0$, $\mu \ge 0$, $t\in (0,T)$, $p\in (2,\infty)$, and assume that $v$ is a smooth solution of
 	the heat equation
 	\begin{align}
 		\begin{split}
 			(\partial_t - \nu\Delta)v &= f,
 			\inon{ in $\Omega_t$},
 			\\
 			v(0)&=v_0,
 			\inon{ in $\Omega$}
 			,\llabel{heat}
 		\end{split}
 	\end{align}
 	with either the Dirichlet boundary condition
 	\begin{align}
 		v &= 0,
 		\inon{ on $\partial \Omega_t$},
 		\llabel{dbdrt}
 	\end{align}
 	the Neumann boundary condition
 	\begin{align}
 		\partial_z v &= g,
 		\inon{ on $\partial \Omega_t$},
 		\llabel{nbdry}
 	\end{align}
 	or the Robin boundary condition
 	\begin{align}
 		\partial_z v -2\mu v  &= 0,
 		\inon{ on $\partial \Omega_t$},
 		\llabel{rbdry}
 	\end{align}
 	where $f$, $g$, and $v_0$ 
 	are given smooth functions.
 	Then, we have the inequality
 	\begin{align}
 		\begin{split}
 			\Vert D^2_x v\Vert&_{L^{p}(\Omega_t)}
 			+\nu^{-\frac{1}{2}}\Vert \nabla v\Vert_{L^{p}(\Omega_t)}
 			+\nu^{-1}(\Vert v\Vert_{L^{p}(\Omega_t)}+
 			\Vert \partial_t v\Vert_{L^{p}(\Omega_t)})
 			\\&\lec
 			\nu^{-1}\Vert f\Vert_{L^{p}(\Omega_t)}
 			+\nu^{-\frac{1}{p}}[\nabla v_0]_{1-\frac{2}{p},p,x,\Omega}
 			+\nu^{-\frac{1}{2}}\Vert \nabla v_0\Vert_{L^{p}}
 			+\nu^{-1}\Vert v_0\Vert_{L^{p}}
 			+I_{\partial \Omega}
 			,\label{maxheat}
 		\end{split}
 	\end{align}
 	where $I_{\partial \Omega}$ equals zero for the Dirichlet and Robin boundary conditions and
 	\begin{align}
 		I_{\partial \Omega}
 		=[g]_{1-\frac{1}{p},x,p,\partial \Omega_t} 
 		+\nu^{\frac{1-p}{2p}}
 		([g]_{\frac{1}{2}-\frac{1}{2p},t,p,\partial \Omega_t} + \Vert g\Vert_{L^{p}(\partial \Omega_t)})
 		,\llabel{EQneumann}
 	\end{align}
 	for the Neumann boundary condition.
 \end{Lemma}
 \colb
 
 We note that $[f]_{s',x,p,\Omega}$
 stands for the spatial Gagliardo-Nirenberg seminorm of $f$.
 In addition, we note that Lemma~\ref{Lh} implicitly 
 assumes the compatibility conditions on the initial data
 \begin{align}
 	v_0 = 0, \inon{ on $\partial \Omega$}
 	,\llabel{cc1}
 \end{align}
 for the Dirichlet boundary condition,
 \begin{align}
 	\partial_z v_0 = g(0), \inon{ on $\partial \Omega$}
 	,\llabel{cc2}
 \end{align}
 for the Neumann boundary condition, and
 \begin{align}
 	\partial_z v_0 - 2\mu v_0 = 0, \inon{ on $\partial \Omega$}
 	,\llabel{cc3}
 \end{align}
 for the Robin boundary condition.
 When $\mu <0 $, we apply Lemma~\ref{Lh} with Dirichlet and Neumann boundary conditions
 to the smooth
 solutions of the Navier-Stokes equations. In this case, we write $g = 2\mu u_\hh$.
 When $\mu \ge 0$, we employ the same lemma with the Dirichlet and Robin boundary conditions.
 In both cases, the compatibility conditions for utilizing Lemma~\ref{Lh} read as in
 the last two equations in \eqref{inidata}.
 For the rest of this section, we assume that \eqref{inidata} holds, and we 
 justify these assumptions by 
 an approximation argument in Section~\ref{sec.main}. 
 
 \begin{proof}[Proof of Lemma~\ref{Lh}]
 	Let $0 < t \le T, \nu>0$ and $\mu \ge 0$ be fixed. Rescaling $v$, $f$, $g$ and $v_0$ as
 	\begin{align}
 		\bar{v}(x,s)=\frac{1}{\sqrt{\nu}}v(\sqrt{\nu}x,s)
 		\comma
 		\bar{f}(x,s)=\frac{1}{\sqrt{\nu}}f(\sqrt{\nu}x,s)
 		\comma
 		\bar{v}_0(x) = \frac{1}{\sqrt{\nu}}v_0(\sqrt{\nu}x)
 		\comma
 		\bar{g}(x_\hh,s) = g(\sqrt{\nu}x_\hh,s)
 		.\llabel{EQM00}
 	\end{align}
 	The function $\bar{v}$ solves
 	\begin{align}
 		\begin{split}
 			(\partial_t - \Delta)\bar{v} &= \bar{f},
 			\inon{ in $\Omega_t$},
 			\\
 			\bar{v}(0)&=\bar{v}_0,
 			\inon{ in $\Omega$}
 			,\llabel{heatrescaled}
 		\end{split}
 	\end{align}
 	either with the Dirichlet boundary condition
 	\begin{align}
 		\bar{v} = 0,
 		\inon{ on $\partial \Omega_t$}
 		,\llabel{EQM01}
 	\end{align}
 	or with the Neumann boundary condition 
 	\begin{align}
 		\partial_z \bar{v} &= \bar{g},
 		\inon{ on $\partial \Omega_t$}
 		.\llabel{EQM002}
 	\end{align}
 	or with the Robin boundary condition 
 	\begin{align}
 		\partial_z \bar{v} -\frac{2\mu}{\sqrt{\nu}}\bar{v} &= 0,
 		\inon{ on $\partial \Omega_t$}
 		.\llabel{EQM02}
 	\end{align}
 	Using the explicit representation of $\bar{v}$, 
 	we arrive at
 	\begin{align}
 		\Vert \bar{v}\Vert_{W_{t,x}^{2,p}(\Omega_t)}
 		\lec
 		\Vert \bar{f}\Vert_{L^{p}(\Omega_t)}
 		+\Vert \bar{v}_0\Vert_{W^{2-\frac{2}{p},6}(\Omega)}
 		+\Vert \bar{g}\Vert_{W^{1-\frac{1}{p}}(\partial \Omega)}
 		,\label{maxheatrescaled}
 	\end{align}
 	where the implicit constant is independent of~$\nu$
 	and $\bar{g}=0$ for the Dirichlet and Robin boundary conditions.
 	We refer the reader to~\cite[Chapter IV]{LSU} as well as \cite{DHP, KM} for 
 	the proof of \eqref{maxheatrescaled}.
 	Now, we may explicitly compute the norms in \eqref{maxheatrescaled},
 	concluding \eqref{maxheat}.
 \end{proof}
 
 In the rest of this section, we estimate each term having $\nu$
 as a factor on the right-hand sides of 
 \eqref{EQ.Con}, \eqref{EQ.Nor}, \eqref{EQ.Nor2}, \eqref{EQ.Pre}, and \eqref{EQ.Inf2}.
 First, we consider the terms involving $u$.
 
 \cole
 \begin{lemma}\label{Lmax1}
 	Let $\mu \in \mathbb{R}$, $\nu \in (0,\bar{\nu}]$, $p \in (2,\infty)$, and assume that $(u,p)$ is a smooth solution of \eqref{NSE0}, \eqref{hnavierbdry} on~$[0,T]$ with a smooth initial datum $u_0$. Then,  we have
 	\begin{align}
 		\begin{split}
 			&\sum_{0\le |\alpha|\le j} \left(\nu \Vert D^2_x Z^\alpha u\Vert_{L^{p}(\Omega_t)}+	
 			\nu^\frac12 \Vert \nabla Z^\alpha u\Vert_{L^{p}(\Omega_t)}
 			+\Vert \partial_t Z^\alpha u\Vert_{L^{p}(\Omega_t)}\right)
 			\lec \mathcal{M}_{j,p} + \mathcal{M}_{0,j,p} + \mathcal{M}_{\mu,j,p}
 			\\&\indeq=
 			\left( \int_0^t 
 			\biggl(\Vert u\Vert_{j+1,p}^p(\Vert u\Vert_{2,\infty}^p+\Vert \nabla u\Vert_{L^{\infty}}^p+1)
 			+ \Vert \nabla p\Vert_{j,p}^p\biggr)\,ds
 			\right)^\frac{1}{p}
 			\\&\indeq \indeq
 			+\sum_{0\le |\alpha|\le j}
 			\left(\nu^{\frac{p-1}{p}}[\nabla Z^\alpha u(0)]_{1-\frac{2}{p},p,x,\Omega}
 			+\nu^{\frac{1}{2}}\Vert \nabla Z^\alpha u(0)\Vert_{L^{p}}
 			+\Vert Z^\alpha u(0)\Vert_{L^{p}}
 			\right)
 			\\&\indeq \indeq
 			+\mu(\sgn(\mu)-1)\sum_{0\le |\alpha|\le j}
 			\left(\nu[Z^\alpha u]_{1-\frac{1}{p},x,p,\partial \Omega_t} 
 			+\nu^{\frac{1+p}{2p}}
 			\left([Z^\alpha u]_{\frac{1}{2}-\frac{1}{2p},t,p,\partial \Omega_t} + \Vert Z^\alpha u\Vert_{L^{p}(\partial \Omega_t)}\right)
 			\right)
 			,\llabel{max1}
 		\end{split}
 	\end{align}
 	for $j=1,2$ and $t \in [0,T]$.
 \end{lemma}
 \colb
 
 We use Lemma~\ref{Lmax1} with $j=2$
 for the conormal derivative and the pressure estimates, while
 the case $j=1$ is needed for the $L^\infty$ estimates when~$\mu \ge 0$.
 
 \begin{proof}[Proof of Lemma~\ref{Lmax1}]
 	We only present the estimates when $j=2$. We recall that $Z^\alpha u_i$ solves
 	\begin{align}
 		(\partial_t -\nu\Delta)Z^\alpha u_i
 		=
 		-Z^\alpha(u\cdot \nabla u_i)-Z^\alpha \partial_i p + \nu(Z^\alpha \Delta -\Delta Z^\alpha )u_i,
 		\llabel{EQ0M24}
 	\end{align}
 	with the boundary conditions depending on $i=1,2,3$, and $\alpha \in \mathbb{N}^3_0$.
 	Now, we consider $Z^\alpha u_i = Z^{\tilde{\alpha}}_\hh Z_3^k u_i$ 
 	for either $i=3$ and $|\alpha|\le 2$, or $i=1,2$, $|\alpha|\le 2$ and $1 \le k \le 2$.  
 	Invoking Lemma~\ref{Lh}
 	with the Dirichlet boundary condition, we obtain
 	\begin{align}
 		\begin{split}
 			\nu\Vert D^2_x Z^\alpha &u_i\Vert_{L^{p}(\Omega_t)}
 			+\nu^{\frac{1}{2}}\Vert \nabla Z^\alpha u_i\Vert_{L^{p}(\Omega_t)}
 			+\Vert Z^\alpha u_i\Vert_{L^{p}(\Omega_t)}+
 			\Vert \partial_t Z^\alpha u_i\Vert_{L^{p}(\Omega_t)}
 			\\&\lec
 			\Vert Z^\alpha (u\cdot \nabla u_i) \Vert_{L^{p}(\Omega_t)}
 			+\Vert Z^\alpha \partial_i p\Vert_{L^{p}(\Omega_t)}
 			+\nu \Vert (Z^\alpha\Delta - \Delta Z^\alpha) u_i\Vert_{L^{p}(\Omega_t)}
 			\\&\indeq
 			+\nu^{\frac{p-1}{p}}[\nabla Z^\alpha u_i(0)]_{1-\frac{2}{p},p,x,\Omega}
 			+\nu^{\frac{1}{2}}\Vert \nabla Z^\alpha u_i(0)\Vert_{L^{p}}
 			+\Vert Z^\alpha u_i(0)\Vert_{L^{p}}
 			,\label{EQM25}
 		\end{split}
 	\end{align}
 	for $t \in [0,T]$.
 	Next, we consider $Z^{\tilde{\alpha}}_\hh Z_3^k u_i$ for $i=1,2$, $k=0$ and $|\tilde{\alpha}|\le 2$. 
 	Employing Lemma~\ref{Lh}
 	with Neumann and Robin boundary conditions when
 	$\mu$ is negative and non-negative, respectively, we conclude 
 	\begin{align}
 		\begin{split}
 			\nu&\Vert D^2_x Z^\alpha u_i\Vert_{L^{p}(\Omega_t)}
 			+\nu^{\frac{1}{2}}\Vert \nabla Z^\alpha u_i\Vert_{L^{p}(\Omega_t)}
 			+\Vert Z^\alpha u_i\Vert_{L^{p}(\Omega_t)}+
 			\Vert \partial_t Z^\alpha u_i\Vert_{L^{p}(\Omega_t)}
 			\\&\lec
 			\Vert Z^\alpha(u\cdot \nabla u_i) \Vert_{L^{p}(\Omega_t)}
 			+\Vert Z^\alpha \partial_i p\Vert_{L^{p}(\Omega_t)}
 			\\&\indeq
 			+\nu^{\frac{p-1}{p}}[\nabla Z^\alpha u_i(0)]_{1-\frac{2}{p},p,x,\Omega}
 			+\nu^{\frac{1}{2}}\Vert \nabla Z^\alpha u_i(0)\Vert_{L^{p}}
 			+\Vert Z^\alpha u_i(0)\Vert_{L^{p}}
 			\\&\indeq
 			+\mu(\sgn(\mu)-1)
 			\left(\nu[Z^\alpha u_i]_{1-\frac{1}{p},x,p,\partial \Omega_t} 
 			+\nu^{\frac{1+p}{2p}}
 			\left([Z^\alpha u_i]_{\frac{1}{2}-\frac{1}{2p},t,p,\partial \Omega_t} + \Vert Z^\alpha u_i\Vert_{L^{p}(\partial \Omega_t)}\right)
 			\right)
 			,\label{EQM26}
 		\end{split}
 	\end{align}
 	for $t \in [0,T]$.
 	We now estimate the advection term and the commutator for the Laplacian. Upon conormalizing when necessary, we have
 	\begin{align}
 		\Vert Z^\alpha (u \cdot \nabla u_i)\Vert_{L^{p}(\Omega_t)}
 		\lec 
 		\left( \int_0^t 
 		\Vert u\Vert_{3,p}^p(\Vert u\Vert_{2,\infty}^p+\Vert \nabla u\Vert_{L^{\infty}}^p) \,ds
 		\right)^\frac{1}{p}
 		,\llabel{EQM27}
 	\end{align}
 	for $|\alpha|\le 2$ and $i =1,2,3$,
 	and
 	\begin{equation}
 		\llabel{EQM28}
 		\nu \Vert (Z^\alpha\Delta - \Delta Z^\alpha) u_i\Vert_{L^{p}(\Omega_t)}
 		\lec 
 		\begin{cases}
 			\nu\Vert D^2_x u_i\Vert_{L^{p}(\Omega_t)}+\nu\Vert \nabla u_i\Vert_{L^{p}(\Omega_t)}, & k=1, |\alpha|=1
 			\\
 			\nu\Vert D^2_x Z_\hh u_i\Vert_{L^{p}(\Omega_t)}+\nu\Vert \nabla Z_\hh u_i\Vert_{L^{p}(\Omega_t)}, & k=1, |\alpha|=2
 			\\
 			\nu\Vert D^2_x Z_3 u_i\Vert_{L^{p}(\Omega_t)}+\nu\Vert \nabla Z_3 u_i\Vert_{L^{p}(\Omega_t)}
 			\\ +\nu\Vert D^2_x u\Vert_{L^{p}}(\Omega_t)+\nu\Vert \nabla u\Vert_{L^{p}}(\Omega_t), & k=2, |\alpha|=2,
 		\end{cases}
 	\end{equation}
 	for $i=1,2,3$ and $t \in [0,T]$.
 	Therefore, 
 	we may control the commutator for the Laplacian term by the right-hand sides of \eqref{EQM25} and \eqref{EQM26}
 	by employing an induction argument on $|\alpha|$.
 	This concludes the proof of Lemma~\ref{Lmax1}.
 \end{proof}
 
 Now, we present the maximum regularity properties of $\eta$.
 
 \cole
 \begin{lemma}\label{Lmax2}
 	Let $\mu \in \mathbb{R}$, $\nu \in (0,\bar{\nu}]$, $p \in (2,\infty)$, and assume that $(u,p)$ is a smooth solution of \eqref{NSE0}, \eqref{hnavierbdry} on~$[0,T]$ with a smooth initial datum $u_0$. Then, for $\eta$ defined in \eqref{eta}, we have 
 	\begin{align}
 		\begin{split}
 			\sum_{0\le |\alpha|\le j}& \left(\nu \Vert D^2_x Z^\alpha \eta\Vert_{L^{p}(\Omega_t)}+	
 			\nu^\frac12 \Vert \nabla Z^\alpha \eta\Vert_{L^{p}(\Omega_t)}\right)\lec \mathcal{N}_{j,p}+\mathcal{N}_{0,j,p}
 			\\&=
 			\left( \int_0^t 
 			\biggl((\Vert \eta\Vert_{j+1,p}^p+\Vert u\Vert_{j+1,p}^p)(\Vert u\Vert_{2,\infty}^p+\Vert \nabla u\Vert_{L^{\infty}}^p+1)
 			+ \Vert p\Vert_{j+1,p}^p\biggr)\,ds
 			\right)^\frac{1}{p}
 			\\& \indeq
 			+\sum_{0\le |\alpha|\le j}
 			\left(\nu^{\frac{p-1}{p}}[\nabla Z^\alpha \eta(0)]_{1-\frac{2}{p},p,x,\Omega}
 			+\nu^{\frac{1}{2}}\Vert \nabla Z^\alpha \eta(0)\Vert_{L^{p}}
 			+\Vert Z^\alpha \eta(0)\Vert_{L^{p}}
 			\right)
 			,\llabel{max2}
 		\end{split}
 	\end{align}
 	for $j=0,1$ and $t \in [0,T]$.
 \end{lemma}
 \colb
 
 Recalling that $Z^\alpha \eta|_{\partial \Omega} = 0$, the proof of Lemma~\ref{Lmax2} follows by employing Lemma~\ref{Lh} for the Dirichlet boundary condition and repeating the proof of Lemma~\ref{Lmax1}. 
 We also note that we need Lemma~\ref{Lmax2} with $j=0$ and $j=1$ for $\mu \ge 0$ and $\mu < 0$, respectively.
 Lastly, since we assume that $T\le 1$, the implicit constant in Lemmas~\ref{Lmax1} and~\ref{Lmax2} does not depend on $T$. 
  
 \startnewsection{Trace estimates}{sec.tr}
 
 In the current section, we estimate the boundary terms from Lemma~\ref{Lmax1} that appear only when $\mu <0$.
 For the terms involving spatial derivatives, we have the standard trace inequality.
 
 \cole
 \begin{lemma}\label{Lt1}
 	Let $\nu \in (0,\bar{\nu}]$, $p \in (2,\infty)$, and assume that $(u,p)$ is a smooth solution of \eqref{NSE0}, \eqref{hnavierbdry} on~$[0,T]$. Then, we have 
 	\begin{align}
 		\begin{split}
 			\sum_{0\le |\alpha|\le j}
 			\left(\nu[Z^\alpha u]_{1-\frac{1}{p},x,p,\partial \Omega_t} 
 			+\nu^{\frac{1+p}{2p}}
 			\Vert Z^\alpha u\Vert_{L^{p}(\partial \Omega_t)}
 			\right)
 			\lec
 			\nu^{\frac{1+p}{2p}}
 			\sum_{0\le |\alpha|\le j}
 			\left(\int_0^t \Vert Z^\alpha u\Vert_{W^{1,p}}^p \,ds\right)^\frac{1}{p}
 			,\llabel{tr1}
 		\end{split}
 	\end{align}
 	for $j=1,2$ and $t \in [0,T]$.
 \end{lemma}
 \colb
 
 We treat the term involving a fractional derivative in time by employing the Fundamental Theorem of Calculus.
 
 \cole
 \begin{lemma}\label{Lt2}
 	Let $\nu \in (0,\bar{\nu}]$, $p \in (5,\infty)$, and assume that $(u,p)$ is a smooth solution of \eqref{NSE0}, \eqref{hnavierbdry} on~$[0,T]$. Then, we have 
 	\begin{align}
 		\begin{split}
 			\nu^{\frac{1+p}{2p}}\sum_{0\le |\alpha|\le j}
 			[Z^\alpha u]_{\frac{1}{2}-\frac{1}{2p},t,p,\partial \Omega_t} 
 			\lec 
 			\nu^{\frac{1+p}{2p}}\sum_{0\le |\alpha|\le j}
 			\Vert \partial_z Z^\alpha u\Vert_{L^{p}(\Omega_t)}^\frac{1}{p}
 			\Vert \partial_t Z^\alpha u\Vert_{L^{p}(\Omega_t)}^\frac{p-1}{p}
 			,\label{tr2}
 		\end{split}
 	\end{align}
 	for $j=1,2$ and $t \in [0,T]$.
 \end{lemma}
 \colb
 
 \begin{proof}[Proof of Lemma~\ref{Lt2}]
 	For $0\le t \le T$ and $\alpha \in \mathbb{N}^3_0$ with $|\alpha|\le 2$ given, we let $f = Z^\alpha u_\hh$,
 	and employing a density argument, we assume that $f$ has an extension $\hat{f} \in C_c^\infty(\mathbb{R}^3 \times (-t,2t))$.
 	Then, we consider $\theta, \tau \in (0,t)$ and write
 	\begin{align}
 		\begin{split}
 			|f(x_\hh, 0,\theta)-f(x_\hh, 0,\tau)|^p
 			&\lec 
 			\int_0^z 
 			|\partial_\zeta \hat{f}(x_\hh, \zeta,\theta)| 
 			|\hat{f}(x_\hh, \zeta,\theta)-\hat{f}(x_\hh, \zeta,\tau)|^{p-1} 
 			\,d\zeta
 			\\&\indeq   +\int_0^z 
 			|\partial_\zeta \hat{f}(x_\hh, \zeta,\tau)| 
 			|\hat{f}(x_\hh, \zeta,\theta)-\hat{f}(x_\hh, \zeta,\tau)|^{p-1} 
 			,\label{EQM29}
 		\end{split}
 	\end{align}
 	We denote $s = 1/2- 1/(2p)$ and multiply both sides of \eqref{EQM29} by $|\theta-\tau|^{-1-ps}$.
 	Then, we utilize the Fundamental Theorem of Calculus again, obtaining
 	\begin{align}
 		\begin{split}
 			\frac{|f(x_\hh, 0,\theta)-f(x_\hh, 0,\tau)|^p}{|\theta-\tau|^{1+ps}}
 			&\lec 
 			\int_0^z 
 			\frac{|\partial_\zeta \hat{f}(x_\hh, \zeta,\theta)|}{|t-\tau|^{1+ps}} 
 			\left|
 			\int_\tau^\theta \partial_s \hat{f}(x_\hh, \zeta,s)\,ds\right
 			|^{p-1} 
 			\,d\zeta
 			\\&\indeq +  
 			\int_0^z 
 			\frac{|\partial_\zeta \hat{f}(x_\hh, \zeta,\tau)|}{|\theta-\tau|^{1+ps}} 
 			\left|
 			\int_\tau^\theta \partial_s \hat{f}(x_\hh, \zeta,s)\,ds\right
 			|^{p-1} 
 			\,d\zeta
 			.\llabel{EQM30}
 		\end{split}	
 	\end{align}
 	Upon integrating, we arrive at
 	\begin{align}
 		\begin{split}
 			[f]_{\frac{1}{2}-\frac{1}{2p},t,p,\partial \Omega_t}^p 
 			&\lec 
 			\int_{(-t,2t)^2} \int_{\mathbb{R}^2} \int_\mathbb{R} 
 			\frac{|\partial_\zeta \hat{f}(x_\hh, \zeta,\theta)|}{|\theta-\tau|^{1+ps}} 
 			\left|
 			\int_\tau^\theta \partial_s \hat{f}(x_\hh, \zeta,s)\,ds\right
 			|^{p-1} 
 			\,d\zeta dx_\hh d\theta d\tau
 			.\llabel{EQM31}
 		\end{split}	
 	\end{align}
 	Now, we employ Hölder's inequality in the $\zeta$ variable and write
 	\begin{align}
 		\begin{split}
 			&[f]_{\frac{1}{2}-\frac{1}{2p},t,p,\partial \Omega_t}^p 
 			\\&\indeq\lec 
 			\int_{(-t,2t)^2} \int_{\mathbb{R}^2} 
 			\left(\int_\mathbb{R}
 			|\partial_\zeta \hat{f}(x_\hh, \zeta,\theta)|^p \,d\zeta\right)^\frac{1}{p}
 			\frac{1}{|\theta-\tau|^{1+ps}} 
 			\left(\int_\mathbb{R}
 			\left|
 			\int_\tau^\theta \partial_s \hat{f}(x_\hh, \zeta,s)\,ds\right
 			|^{p} \,d\zeta\right)^\frac{p-1}{p}  
 			\,dx_\hh d\theta d\tau
 			,\llabel{EQM32}
 		\end{split}	
 	\end{align}
 	from where we utilize Hölder's inequality in $x_\hh$, $t$ and $\tau$ variables to obtain
 	\begin{align}
 		\begin{split}
 			[f]_{\frac{1}{2}-\frac{1}{2p},t,p,\partial \Omega_t}^p 
 			&\lec 
 			\Vert \partial_z \hat{f}\Vert_{L^{p}((-t,2t)^2 \times \mathbb{R}^3)}
 			\int_{(-t,2t)^2} \frac{1}{|\theta-\tau|^\frac{(1+ps)p}{p-1}}
 			\int_{\mathbb{R}^3} 
 			\left|
 			\int_\tau^\theta \partial_s \hat{f}(x_\hh, \zeta,s)\,ds\right
 			|^{p}  
 			\,d\zeta dx_\hh d\theta d\tau
 			.\llabel{EQM33}
 		\end{split}	
 	\end{align}
 	Finally, we apply Hölder's inequality in the $s$ variable and write
 	\begin{align}
 		\begin{split}
 			[f]_{\frac{1}{2}-\frac{1}{2p},t,p,\partial \Omega_t}^p 
 			&\lec 
 			\Vert \partial_z \hat{f}\Vert_{L^{p}((-t,2t)^2 \times \mathbb{R}^3)}
 			\left(\int_{(-t,2t)^2} \frac{|\theta-\tau|^{p-1}}{|\theta-\tau|^\frac{(1+ps)p}{p-1}}
 			\int_{\mathbb{R}^3} 
 			\int_{-t}^{2t} |\partial_s \hat{f}(x_\hh, \zeta,s)|^p \,ds
 			\,d\zeta dx_\hh d\theta d\tau\right)^\frac{p-1}{p}
 			.\llabel{EQM34}
 		\end{split}	
 	\end{align}
 	Recalling that $t\le T\le 1$, $p>5$, and $s<1$, 
 	it follows that the term involving $|\theta-\tau|$ is bounded. Therefore,
 	we conclude
 	\begin{align}
 		\begin{split}
 			[f]_{\frac{1}{2}-\frac{1}{2p},t,p,\partial \Omega_t}^p 
 			&\lec 
 			\Vert \partial_z \hat{f}\Vert_{L^{p}((-t,2t)^2 \times \mathbb{R}^3)}
 			\Vert \partial_t \hat{f}\Vert_{L^{p}((-t,2t)^2 \times \mathbb{R}^3)}^{p-1}
 			,\llabel{EQM35}
 		\end{split}	
 	\end{align}
 	from where \eqref{tr2} follows by approximation
 	and the continuity of the extension operator.
 \end{proof}

 \startnewsection{Concluding the a~priori estimates}{sec.apri}
 
 We first refine the pressure estimates utilizing Lemma~\ref{Lmax1}.
 Employing Young's inequality for the last term in \eqref{EQ.Pre}, we obtain
 \begin{align}
 	\nu \sum_{1\le |\tilde{\theta}|\le 3}\Vert u(t)\Vert_{3,p}^{\frac{p-1}{p}}
 	\Vert \nabla Z^{\tilde{\theta}}_\hh u(t)\Vert_{L^p}^{\frac{1}{p}}
 	\lec 
 	\nu^p \Vert \nabla Z u(t)\Vert_{2,p} + \Vert u(t)\Vert_{3,p}
 	,\llabel{EQM36}
 \end{align}
 from where, applying the $L^p$ norm in time and Lemma~\ref{Lmax1},
 we arrive at
 \begin{align}
 	\begin{split}
 		\int_0^t  (\Vert D^2 p\Vert_{2,p}^p+
 		\Vert \nabla p\Vert_{2,p}^p) \,ds 
 		&\lec
 		\nu^{p(p-1)} \int_0^t \Vert \nabla p\Vert_{2,p}^p \,ds
 		+ 
 		\int_0^t 
 		\biggl(\Vert u\Vert_{3,p}^p (\Vert u\Vert_{2,\infty}^p+\Vert \eta\Vert_{L^{\infty}}^p+1)
 		\biggr)\,ds
 		\\&\indeq
 		+\nu^{p(p-1)} \mathcal{M}^p_{0,2,p} + \nu^{p(p-1)}\mathcal{M}^p_{\mu,2,p} 
 		,\llabel{EQM37}
 	\end{split}
 \end{align}
 for $t \in [0,T]$.
 Recalling that $\nu \le \bar{\nu}$, we may choose $\bar{\nu}$ sufficiently small
 and absorb the pressure term to conclude
 \begin{align}
 	\begin{split}
 		\int_0^t  (\Vert D^2 p\Vert_{2,p}^p+
 		\Vert \nabla p\Vert_{2,p}^p) \,ds 
 		&
 		\lec
 		\int_0^t 
 		\biggl(\Vert u\Vert_{3,p}^p (\Vert u\Vert_{2,\infty}^p+\Vert \eta\Vert_{L^{\infty}}^p+1)
 		\biggr)\,ds
 		\\&\indeq
 		+\nu^{p(p-1)} \mathcal{M}^p_{0,2,p} + \nu^{p(p-1)}\mathcal{M}^p_{\mu,2,p} 
 		,\label{EQM38}
 	\end{split}
 \end{align}
 for $t \in [0,T]$.
 The same reasoning also yields
 \begin{align}
 	\begin{split}
 		\int_0^t  (\Vert D^2 p\Vert_{1,p}^p+
 		\Vert \nabla p\Vert_{1,p}^p) \,ds 
 		&
 		\lec
 		\int_0^t 
 		\biggl(\Vert u\Vert_{2,p}^p (\Vert u\Vert_{2,\infty}^p+\Vert \eta\Vert_{L^{\infty}}^p+1)
 		\biggr)\,ds
 		\\&\indeq+\nu^{p(p-1)} \mathcal{M}^p_{0,1,p} + \nu^{p(p-1)}\mathcal{M}^p_{\mu,1,p} 
 		,\label{EQM39}
 	\end{split}
 \end{align}
 for $t \in [0,T]$.
 Now, we use \eqref{EQM38} and \eqref{EQM39} to refine the
 maximal regularity estimates given by Lemma~\ref{Lmax1} and Lemma~\ref{Lmax2}.
 For $u$, we obtain
 \begin{align}
 	\begin{split}
 		\sum_{0\le |\alpha|\le j}& \left(\nu^p \Vert D^2_x Z^\alpha u\Vert_{L^{p}(\Omega_t)}^p+	
 		\nu^\frac{p}{2} \Vert \nabla Z^\alpha u\Vert_{L^{p}(\Omega_t)}^p
 		+\Vert \partial_t Z^\alpha u\Vert_{L^{p}(\Omega_t)}^p\right)
 		\\&\lec 
 		\int_0^t 
 		\Vert u\Vert_{j+1,p}^p(\Vert u\Vert_{2,\infty}^p+\Vert \nabla u\Vert_{L^{\infty}}^p+1)
 		\,ds
 		+\mathcal{M}^p_{0,j,p} + \mathcal{M}^p_{\mu,j,p}
 		,\label{max1ref}
 	\end{split}
 \end{align}
 for $j=1,2$ and $t \in [0,T]$.
 Next, for $\eta$, we have
 \begin{align}
 	\begin{split}
 		\sum_{0\le |\alpha|\le j}& \left(\nu^p \Vert D^2_x Z^\alpha \eta\Vert_{L^{p}(\Omega_t)}^p+	
 		\nu^\frac{p}{2} \Vert \nabla Z^\alpha \eta\Vert_{L^{p}(\Omega_t)}^p\right)
 		\\&\lec 
 		\int_0^t 
 		(\Vert \eta\Vert_{j+1,p}^p+\Vert u\Vert_{j+1,p}^p)(\Vert u\Vert_{2,\infty}^p+\Vert \nabla u\Vert_{L^{\infty}}^p+1)
 		\,ds
 		+\mathcal{N}^p_{0,j,p}
 		,\label{max2ref}
 	\end{split}
 \end{align}
 for $j=0,1$ and $t \in [0,T]$.
 Next using Lemmas~\ref{Lt1} and~\ref{Lt2}, we further refine \eqref{max1ref} for the case $\mu <0$.
 In particular, these two lemmas with Young's inequality imply
 \begin{align}
 	\begin{split}
 		\mu^p &\sum_{0\le |\alpha|\le j}
 		\left(\nu^p [Z^\alpha u]_{1-\frac{1}{p},x,p,\partial \Omega_t}^p 
 		+\nu^{\frac{1+p}{2}}
 		\left([Z^\alpha u]_{\frac{1}{2}-\frac{1}{2p},t,p,\partial \Omega_t}^p + \Vert Z^\alpha u\Vert_{L^{p}(\partial \Omega_t)}^p\right)
 		\right)
 		\\& \lec
 		\int_0^t \Vert u\Vert_{2,p}^p \,ds
 		+
 		\nu^{\frac{1+p}{2}}
 		\sum_{0\le |\alpha|\le j}
 		\int_0^t \left(\Vert \nabla Z^\alpha u\Vert_{L^{p}}^p 
 		+ \Vert \partial_t Z^\alpha u\Vert_{L^{p}}
 		\right) \,ds
 		,\label{EQM42}     
 	\end{split}
 \end{align}
 for $j=1,2$, $p>5$, and $t \in [0,T]$.
 Recalling that $\nu \le \bar{\nu}$ and $\bar{\nu}$ is sufficiently small,
 \eqref{max1ref} implies
 \begin{align}
 	\begin{split}
 		\sum_{0\le |\alpha|\le j}& \left(\nu \Vert D^2_x Z^\alpha u\Vert_{L^{p}(\Omega_t)}^p+	
 		\nu^\frac12 \Vert \nabla Z^\alpha u\Vert_{L^{p}(\Omega_t)}^p
 		+\Vert \partial_t Z^\alpha u\Vert_{L^{p}(\Omega_t)}^p\right)
 		\\&\lec 
 		\int_0^t 
 		\Vert u\Vert_{j+1,p}^p(\Vert u\Vert_{2,\infty}^p+\Vert \nabla u\Vert_{L^{\infty}}^p+1)
 		\,ds
 		+\mathcal{M}^p_{0,j,p}
 		,\label{max1refref}
 	\end{split}
 \end{align}
 for $j=1,2$ and $t \in [0,T]$.
 
 Now, we proceed to establish \eqref{ap1}. Recalling $P$ from \eqref{pol1},
 we let $p=3+\delta$ and employ \eqref{EQ.Con}, \eqref{EQ.Nor}, and \eqref{EQ.Nor2}.
 Then, we take $p=6+2\delta$ and utilize \eqref{EQ.Con} again. Finally, 
 we add the resulting inequalities with \eqref{EQ.Inf1} obtaining
 \begin{align}
 	\begin{split}
 		\Vert u&\Vert_{3,3+\delta}^{3+\delta}
 		+\Vert u\Vert_{3,6+2\delta}^{6+2\delta}
 		+\Vert \eta\Vert_{2,3+\delta}^{3+\delta}
 		+\Vert \eta\Vert_{L^{\infty}}
 		\\&\lec
 		P_0
 		+\int_0^t 
 		(\Vert u\Vert_{2,\infty}
 		+\Vert u\Vert_{W^{1,\infty}} 
 		+1)P \,ds
 		+\int_0^t
 		(\Vert \nabla p\Vert_{3,3+\delta}^{3+\delta}
 		+\Vert \nabla p\Vert_{3,6+2\delta}^{6+2\delta}
 		+\Vert Z_\hh p\Vert_{L^\infty})\,ds
 		\\&\indeq
 		+\nu^{\frac{3+\delta}{2}}\int_0^t (\Vert \partial_z u\Vert_{2,3+\delta}^{3+\delta}
 		+\Vert \partial_z \eta\Vert_{1,3+\delta}^{3+\delta})\,ds
 		+\nu^{3+\delta}\int_0^t \Vert \partial_z u\Vert_{2,6+2\delta}^{6+2\delta}\,ds
 		,\label{EQM40}
 	\end{split}
 \end{align}
 for $t \in [0,T]$.
 To estimate the pressure terms in \eqref{EQM40}, we use the embedding $\Vert Z_\hh p\Vert_{L^\infty} \lec \Vert \nabla p\Vert_{1,3+\delta}$,
 the inequality
 \begin{align}
 	\Vert \nabla p\Vert_{3,3+\delta}
 	\lec 
 	\Vert \nabla p\Vert_{L^{3+\delta}}+\Vert D^2_x p\Vert_{2,6+2\delta}
 	,\llabel{EQM41}
 \end{align}
 Lemma~\ref{L.Pre} with $(j,p)=(0,3+\delta)$, and
 the pressure estimates \eqref{EQM38} with $p=6+2\delta$, and $T\le 1$. Consequently, 
 \begin{align}
 	\begin{split}
 		\Vert u&\Vert_{3,3+\delta}^{3+\delta}
 		+\Vert u\Vert_{3,6+2\delta}^{6+2\delta}
 		+\Vert \eta\Vert_{2,3+\delta}^{3+\delta}
 		+\Vert \eta\Vert_{L^{\infty}}
 		\\&\lec
 		P_0 
 		+ \nu^{(6+2\delta)(5+2\delta)}\mathcal{M}^{6+2\delta}_{0,2,6+2\delta}
 		+ \nu^{(6+2\delta)(5+2\delta)} \mathcal{M}^{6+2\delta}_{\mu,2,6+2\delta}
 		+\int_0^t 
 		(\Vert u\Vert_{2,\infty}^{6+2\delta}
 		+\Vert u\Vert_{W^{1,\infty}}^{6+2\delta} 
 		+1)P \,ds
 		\\&\indeq
 		+\nu^{\frac{3+\delta}{2}}\int_0^t (\Vert \partial_z u\Vert_{2,3+\delta}^{3+\delta}
 		+\Vert \partial_z \eta\Vert_{1,3+\delta}^{3+\delta})\,ds
 		+\nu^{3+\delta}\int_0^t \Vert \partial_z u\Vert_{2,6+2\delta}^{6+2\delta}\,ds
 		,\llabel{EQM43}
 	\end{split}
 \end{align}
 for $t \in [0,T]$.
 Recalling $\mathcal{M}^{6+2\delta}_{\mu,2,6+2\delta}$ from Lemma~\ref{Lmax1}, 
 we estimate it employing \eqref{EQM42} and write
 \begin{align}
 	\begin{split}
 		\Vert u&\Vert_{3,3+\delta}^{3+\delta}
 		+\Vert u\Vert_{3,6+2\delta}^{6+2\delta}
 		+\Vert \eta\Vert_{2,3+\delta}^{3+\delta}
 		+\Vert \eta\Vert_{L^{\infty}}
 		\\&\lec
 		P_0 
 		+ \nu^{(6+2\delta)(5+2\delta)}\mathcal{M}^{6+2\delta}_{0,2,6+2\delta}
 		+\int_0^t 
 		(\Vert u\Vert_{2,\infty}^{6+2\delta}
 		+\Vert u\Vert_{W^{1,\infty}}^{6+2\delta} 
 		+1)P \,ds
 		\\&\indeq
 		+\nu^{\frac{3+\delta}{2}}\int_0^t (\Vert \partial_z u\Vert_{2,3+\delta}^{3+\delta}
 		+\Vert \partial_z \eta\Vert_{1,3+\delta}^{3+\delta})\,ds
 		+\nu^{3+\delta}\int_0^t (\Vert \partial_z u\Vert_{2,6+2\delta}^{6+2\delta}
 		+\Vert \partial_t u\Vert_{2,6+2\delta}^{6+2\delta})\,ds
 		,\llabel{EQM44}
 	\end{split}
 \end{align}
 where we have used that $(6+2\delta)(5+2\delta)>3+\delta$.
 Before invoking the maximal regularity estimates, we use \eqref{EQ62} obtaining
 \begin{align}
 	\int_0^t \Vert \partial_z u\Vert_{2,3+\delta}^{3+\delta} \,ds
 	\lec
 	\int_0^t \Vert \eta\Vert_{2,3+\delta}^{3+\delta} +\Vert u\Vert_{3,3+\delta}^{3+\delta} \,ds
 	\lec
 	\int_0^t P \,ds
 	.\llabel{EQM45}
 \end{align}
 Now, we employ \eqref{max1refref} for $(p,j)=(6+2\delta,2)$
 and \eqref{max2ref} for $(p,j)=(3+\delta,1)$, from where
 it follows that
 \begin{align}
 	\begin{split}
 		\Vert u\Vert_{3,3+\delta}^{3+\delta}
 		&+\Vert u\Vert_{3,6+2\delta}^{6+2\delta}
 		+\Vert \eta\Vert_{2,3+\delta}^{3+\delta}
 		+\Vert \eta\Vert_{L^{\infty}}
 		\\&\lec
 		P_0 
 		+ \mathcal{M}^{6+2\delta}_{0,2,6+2\delta}
 		+\mathcal{N}_{0,1,3+\delta}^{3+\delta}
 		+\int_0^t 
 		(\Vert u\Vert_{2,\infty}^{6+2\delta}
 		+\Vert u\Vert_{W^{1,\infty}}^{6+2\delta} 
 		+1)P \,ds
 		.\llabel{EQM46}
 	\end{split}
 \end{align}
 Utilizing \eqref{EQ62} for the Lipschitz norm of $u$
 and the inequality
 \begin{align}
 	\Vert u\Vert_{2,\infty}
 	\lec
 	\Vert \nabla u\Vert_{2,3+\delta}+\Vert u\Vert_{2,3+\delta}
 	\lec
 	\Vert \eta\Vert_{2,3+\delta}+\Vert u\Vert_{3,3+\delta}
 	,\llabel{EQM47}
 \end{align}
 we arrive at
 \begin{align}
 	\Vert u\Vert_{3,3+\delta}^{3+\delta}
 	+\Vert u\Vert_{3,6+2\delta}^{6+2\delta}
 	+\Vert \eta\Vert_{2,3+\delta}^{3+\delta}
 	+\Vert \eta\Vert_{L^{\infty}}
 	\lec
 	P_0 
 	+ \mathcal{M}^{6+2\delta}_{0,2,6+2\delta}
 	+\mathcal{N}_{0,1,3+\delta}^{3+\delta}
 	+\int_0^t 
 	P \,ds
 	.\llabel{EQM48}
 \end{align}
 Finally, the $L^2$ estimates follow from standard considerations.
 Indeed, we may use incompressibility to eliminate the pressure term vanishes and
 the trace and Young's inequality to absorb the boundary term,
 and \eqref{ap1} follows.
 
 Next, we establish \eqref{ap2}. To achieve this, we employ 
 \eqref{EQ.Con}, and \eqref{EQ.Nor} for $p=6$.
 Recalling $Q$ from \eqref{pol1}, we utilize \eqref{EQ.Inf2} and add the resulting inequalities
 obtaining
 \begin{align}
 	\begin{split}
 		\Vert u\Vert_{3,6}^{6}
 		+\Vert \eta\Vert_{1,6}^{6}
 		+\Vert u\Vert_{2,\infty}^{2}
 		+\Vert \eta\Vert_{L^{\infty}}^2
 		&\lec
 		Q_0
 		+\int_0^t Q \,ds
 		+\int_0^t
 		(\Vert \nabla p\Vert_{3,6}^{6}
 		+\Vert \nabla p\Vert_{2,\infty}^2)\,ds
 		\\&\indeq
 		+\nu^{3}\int_0^t (\Vert \partial_z u\Vert_{2,6}^{6}
 		+\Vert \partial_z \eta\Vert_{L^6}^{6})\,ds
 		+\nu\int_0^t \Vert \partial_z u\Vert_{1,\infty}^{2}\,ds
 		,\llabel{EQM50}
 	\end{split}
 \end{align}
 for $t \in [0,T]$.
 We note that
 \begin{align}
 	\Vert \nabla p\Vert_{2,\infty} \lec \Vert D^2 p\Vert_{2,6} + \Vert \nabla p\Vert_{2,6}
 	,\llabel{EQM51}
 \end{align}
 from where, using the pressure estimates \eqref{EQM38}, we arrive at
 \begin{align}
 	\begin{split}
 		\Vert u&\Vert_{3,6}^{6}
 		+\Vert \eta\Vert_{1,6}^{6}
 		+\Vert u\Vert_{2,\infty}^{2}
 		+\Vert \eta\Vert_{L^{\infty}}^2
 		\\&
 		\lec
 		Q_0
 		+\int_0^t Q \,ds
 		+\nu^{3}\int_0^t (\Vert \partial_z u\Vert_{2,6}^{6}
 		+\Vert \partial_z \eta\Vert_{L^6}^{6})\,ds
 		+\nu\int_0^t \Vert \partial_z u\Vert_{1,\infty}^{2}\,ds
 		+\nu^{30} \mathcal{M}_{0,2,6}^6
 		,\llabel{EQM50}
 	\end{split}
 \end{align}
 for $t \in [0,T]$.
 Recalling the definition of $\mathcal{M}_{\mu,j,p}$ from Lemma~\ref{Lmax1},
 we note that no boundary terms appear when $\sgn(\mu)=1$.
 Now, for the terms involving $\partial_z u$ we have
 \begin{align}
 	\nu^3\Vert \partial_z u\Vert_{2,6} 
 	\lec
 	\nu^3(\Vert \eta\Vert_{2,6}+\Vert u\Vert_{3,6}
 	),\llabel{EQM51}
 \end{align}
 and 
 \begin{align}
 	\nu\Vert \partial_z u\Vert_{1,\infty}^2
 	\lec
 	\nu\Vert D^2_x u\Vert_{1,6}\Vert \nabla u\Vert_{1,6}
 	\lec
 	\nu^2 \Vert D^2_x u\Vert_{1,6}^2 + Q
 	.\llabel{EQM52} 
 \end{align}
 We conclude \eqref{ap2} upon employing $L^2$ estimates, the refined maximal regularity estimates \eqref{max1ref}
 for $(p,j)=(6,2)$ and $(6,1)$ and \eqref{max2ref} for $(p,j)=(6,0)$,
 as well as using that $\nu \le \bar{\nu}$ sufficiently small and $T \le 1$. 
 
 Finally, \eqref{ap3} follows by performing $L^2$ estimates, employing
 \eqref{EQ.Con} and \eqref{EQ.Pre} for $p=3+\delta$ and $\nu = 0$,
 as well as \eqref{EQ.Inf3} and the inequality
 \begin{align}
 	\Vert \nabla p\Vert_{2,\infty} \lec \Vert D^2 p\Vert_{2,3+\delta} + \Vert \nabla p\Vert_{2,3+\delta}
 	,\llabel{EQM53}
 \end{align}
 which concludes the proof of Proposition~\ref{P.Ap}.
 
 \startnewsection{Proofs of Theorems~\ref{T01},~\ref{T03}, and~\ref{T04}}{sec.main}
 
 In this section, we only present proofs of Theorem~\ref{T01} and Theorem~\ref{T04}
 as we may establish Theorem~\ref{T03} by adjusting the regularity of the initial datum $u_0$.
 To start, we assume that $u_0$ satisfying the assumptions in Theorem~\ref{T01} is given, and 
 we consider the heat equation
 \begin{align}
 	\begin{split}
 		(\partial_t - \Delta) v &= 0,
 		\hspace{0.6cm} \inon{ in $\Omega\times (0,\bar{\nu}]$},
 		\\
 		v_3 = 0, \text{ and } \partial_z v_\hh &= 2\mu v_\hh,
 		\inon{ on $\partial\Omega\times (0,\bar{\nu}]$},
 		\\
 		v(0) &= u_0
 		,\llabel{heateqn}
 	\end{split}	
 \end{align} 
 which has a unique solution that is smooth in positive time. 
 Moreover, upon letting $u_0^\nu(x) = v(\nu^3,x)$, we obtain the
 compatibility conditions
 \begin{align}
 	\begin{split}
 		\nabla \cdot u_0^\nu = 0 \comma
 		(u_0^\nu)_3|_{\partial \Omega} = 0, \text{ and }
 		\partial_z (u_0^\nu)_\hh|_{\partial \Omega} = 2\mu (u_0^\nu)_\hh|_{\partial \Omega},
 		\label{inidata}
 	\end{split}
 \end{align}
 as well as the bounds
 \begin{align}
 	\begin{split}
 		\sup_\nu \left(
 		\Vert u_0^\nu\Vert_{W^{3,6+2\delta}\cap W^{3,3+\delta}}
 		+\Vert \nabla u_0^\nu\Vert_{W^{2,3+\delta}\cap L^{\infty}}\right)
 		&\le C(\Vert u_0\Vert_{W^{3,6+2\delta}\cap W^{3,3+\delta}}
 		+\Vert \nabla u_0\Vert_{W^{2,3+\delta}\cap L^{\infty}}),
 		\\
 		\mathcal{M}_{0,2,6+2\delta}^{6+2\delta}(u_0^\nu)
 		+\mathcal{N}_{0,1,3+\delta}^{3+\delta}(\eta_0^\nu)
 		&\le
 		C\nu
 		,\label{inibds}
 	\end{split} 	
 \end{align}
 and the convergence
 \begin{align}
 	\Vert u_0^\nu - u_0\Vert_{L^{2}} \le C\nu^2 
 	.\llabel{iniconv}
 \end{align}
 Indeed, $u_0^\nu$ converges to $u_0$ in a stronger space. However, the $L^2$ convergence 
 suffices to establish the inviscid limit.
 Now, under standard considerations, \eqref{NSE0} and \eqref{hnavierbdry} with
 the initial datum $u_0^\nu$ has a sufficiently smooth unique solution
 for which \eqref{ap1} holds. Therefore, employing the bounds \eqref{inibds}
 and the Gronwall inequality, we conclude that \eqref{EQ.main2}
 holds on a time interval $[0,T]$ that is independent of $\nu$.  
 Finally, to obtain a limit, we need a compactness result.
 
 \cole
 \begin{Proposition}[Compactness of $\{u^\nu\}_\nu$]
 	\label{P01}
 	Let $\nu \in (0,\bar{\nu}]$, and assume that
 	$u^\nu \in L^\infty(0,T;H^1(\Omega)\cap W^{1,\infty}(\Omega))$
 	is a sequence of solutions
 	to \eqref{NSE0}--\eqref{hnavierbdry} on~$[0,T]$ with the initial data 
 	$u_0^\nu \in L^2$ satisfying $\Vert u^{\nu_1}_0-u^{\nu_2}_0\Vert_{L^{2}} \le \nu_1-\nu_2$.
 	Then, $\{u^\nu\}_\nu$ is a Cauchy sequence in $L^\infty(0,T;L^2(\Omega))$
 	with
 	\begin{align}
 		\sup_{[0,T]}\Vert u^{\nu_1}-u^{\nu_2}\Vert_{L^2}^2
 		\lec \nu_1 + \nu_2,
 		\llabel{cauchy}
 	\end{align}
 	where the implicit constant is independent of $\nu_1$ and~$\nu_2$.
 \end{Proposition}
 \colb
 
 Proposition~\ref{P01} follows upon performing $L^2$
 estimates on $u^{\nu_1}-u^{\nu_2}$.
 We note that~\cite{AK1} has a similar result which requires $\nu D^2_x u^\nu \in L^2(0,T;L^2)$. However, we do not need
 this assumption since we may justify integration-by-parts in the distributional sense.
 In addition, $u^\nu\in L^\infty(0,T;H^1\cap W^{1,\infty})$
 follows from the interpolation inequality
 \begin{align}
 	\Vert \nabla u^\nu\Vert_{L^2}
 	\lec
 	\Vert u^\nu\Vert_{L^2}^\frac{3}{5}
 	\Vert \nabla u^\nu\Vert_{L^{\infty}}^\frac{2}{5}
 	+\Vert u^\nu\Vert_{L^{2}}
 	.\llabel{EQM43}
 \end{align}
 Now, we may utilize the strong convergence given by Proposition~\ref{P01}
 and the weak and weak* convergences by \eqref{EQ.main2}. Upon passing to a subsequence, we obtain a solution $u$ for the Euler equations.
 Moreover, since we have Lipschitz regularity, this solution is unique. 
 Finally, \eqref{invlimit1} follows from
 \begin{align}
 	\sup_{[0,T]}\Vert u^\nu-u\Vert_{L^{p}}
 	\lec
 	\sup_{[0,T]}(\Vert u^\nu-u\Vert_{L^{2}}^\frac{2p+6}{5p}\Vert \nabla u^\nu-u\Vert_{L^{\infty}}^\frac{3p-6}{5p}+\Vert u^\nu -u\Vert_{L^{2}}
 	)\lec
 	\nu^{\frac{p+3}{5p}},
 	\llabel{EQM44}
 \end{align}
 concluding the proof of Theorem~\ref{T01}.
 
 Next, the proof of Theorem~\ref{T04} follows closely
 to the one presented in~\cite{AK2}.
 We consider $u_0$ as in Theorem~\ref{T04}, and we let $\{u_0^r\}_{r>0} \in C^\infty(\Omega)$ be a sequence of divergence-free 
 smooth functions that are tangential on the boundary. 
 In particular, we may assume that $u_0^r \in H^5$ and
 \begin{align}
 	\begin{split}
 		u_0^r &\to u_0 \text{ strongly in } L^2(\Omega) \cap W^{3,3+\delta}_\cco(\Omega),
 		\\
 		u_0^r &\rightharpoonup u_0 \text{ weakly-* in } W^{1,\infty}(\Omega)\cap W^{2,\infty}_\cco(\Omega),
 		\llabel{EQ306}  
 	\end{split}
 \end{align}
 as $r \to 0$.
 Therefore, there is a unique $u^r \in C([0,T];H^5(\Omega))$
 that solves \eqref{euler} and \eqref{eulerb} with the initial datum $u_0^r$, where $T>0$.
 We note that $T$ is independent of $r$, and this is due to the control on the Lipschitz norm 
 established by a~priori estimates. Indeed, we may continue the solutions $u^r$
 as long as $\nabla u^r$ stays in $L^1(0,T;L^\infty)$. Employing 
 the a~priori estimates, we conclude that  
 $u^r \in L^\infty(0,T_0;L^2 \cap W^{3,3+\delta}_\cco \cap W^{1,\infty}\cap W^{2,\infty}_\cco)$
 are bounded independent of~$r$.
 Next, following the arguments in~\cite{AK2}, we may 
 prove that $u^r \in L^\infty(0,T_0;L^2(\Omega))$ is a Cauchy sequence.
 Finally, by passing to a subsequence, we may conclude that there exists 
 $u \in L^\infty(0,T_0;L^2 \cap W^{3,3+\delta}_\cco\cap W^{1,\infty} \cap W^{2,\infty}_\cco)$, 
 a solution for \eqref{euler} and \eqref{eulerb}. In addition, this solution is unique since it is Lipschitz continuous.
 
 \colb
 \section*{Acknowledgments}
 The author was supported in part by the
 NSF grant DMS-2205493 and is grateful to Igor Kukavica for helpful discussions.


\begin{thebibliography}{[GKLMN]}
\small

\bibitem[ACS]{ACS}
A.~Argenziano, M.~Cannone, and  M.~Sammartino,
\emph{Navier-{S}tokes equations in the half space with non
	compatible data},
J. Math. Fluid Mech.,~\textbf{26} (2024), no.~2, Paper No. 32, 40.

\bibitem[AK1]{AK1}  
M.S.~Ayd\i n, and I.~Kukavica,    
\emph{Uniform bounds and the inviscid limit for the Navier-Stokes equations with Navier boundary conditions}, 
	arXiv:2404.17111.
	
\bibitem[AK2]{AK2}  
M.S.~Ayd\i n, and I.~Kukavica,    
\emph{Euler Equations in Sobolev conormal spaces}, 
	arXiv:2407.18149.

\bibitem[BB]{BB}
J. P.~Bourguignon, and  H.~Brezis,
\emph{Remarks on the {E}uler equation},
J. Functional Analysis,~\textbf{15} (1974), 341--363.

\bibitem[BdVC1]{BdVC1}
H.~Beir\~{a}o da Veiga, and F.~Crispo,
\emph{Sharp inviscid limit results under {N}avier type boundary
	conditions. {A}n {$L^p$} theory},
J. Math. Fluid Mech.,~\textbf{12} (2010), no.~3, 397--411.

\bibitem[BdVC2]{BdVC2}
H.~Beir\~{a}o da Veiga, and F.~Crispo,
\emph{Concerning the {$W^{k,p}$}-inviscid limit for 3-{D} flows
under a slip boundary condition},
J. Math. Fluid Mech.,~\textbf{13} (2011), no.~1, 117--135.

\bibitem[BdVC3]{BdVC3}
H.~Beir\~{a}o da Veiga, and F.~Crispo,
\emph{A missed persistence property for the {E}uler equations and
	its effect on inviscid limits},
Nonlinearity,~\textbf{25} (2012), no.~6, 1661--1669.

\bibitem[BdVC4]{BdVC4}
H.~Beir\~{a}o da Veiga, and F.~Crispo,
\emph{The 3-{D} inviscid limit result under slip boundary
	conditions. {A} negative answer},
J. Math. Fluid Mech.,~\textbf{14} (2012), no.~1, 55--59.

\bibitem[BILN]{BILN} 
A.V.~Busuioc, D.~Iftimie, M.C.~Lopes~Filho, and H.J.~Nussenzveig~Lopes,
\emph{Uniform time of existence for the alpha {E}uler equations},
J. Funct. Anal.~\textbf{271} (2016), no.~5, 1341--1375.

\bibitem[BS1]{BS1}
L.C.~Berselli, and S.~Spirito,
\emph{On the vanishing viscosity limit of 3{D} {N}avier-{S}tokes
	equations under slip boundary conditions in general domains},
Comm. Math. Phys.,~\textbf{316} (2012), no.~1, 171--198.

\bibitem[BS2]{BS2}
L.C.~Berselli, and S.~Spirito,
\emph{An elementary approach to the inviscid limits for the 3{D}
	{N}avier-{S}tokes equations with slip boundary conditions and
	applications to the 3{D} {B}oussinesq equations},
NoDEA Nonlinear Differential Equations Appl.,~\textbf{21} (2014), no.~2, 149--166.

\bibitem[C1]{C1}
D.~Chae,
\emph{On the well-posedness of the {E}uler equations in the
	{T}riebel-{L}izorkin spaces},
Comm. Pure Appl. Math.,~\textbf{55}, no.~5, (2002), 654--678.

\bibitem[C2]{C2}
D.~Chae,
\emph{On the {E}uler equations in the critical {T}riebel-{L}izorkin
spaces},
Arch. Ration. Mech. Anal.,~\textbf{170}, no.~3, (2003), 185--210.

\bibitem[C3]{C3}
D.~Chae,
\emph{Local existence and blow-up criterion for the {E}uler
equations in the {B}esov spaces},
Asymptot. Anal.,~\textbf{38}, no.~3-4, (2004), 339--358.

\bibitem[CMR]{CMR} 
T.~Clopeau, A.~Mikeli\'c, and R.~Robert,
\emph{On the vanishing viscosity limit for the {$2{\rm D}$}
	incompressible {N}avier-{S}tokes equations with the friction
	type boundary conditions}, Nonlinearity~\textbf{11} (1998), no.~6, 1625--1636.


\bibitem[CLNV]{CLNV} 
P.~Constantin, M.C. Lopes~Filho, H.J. Nussenzveig~Lopes, and V. Vicol,
\emph{Vorticity measures and the inviscid limit}, Arch. Ration. Mech. Anal.~\textbf{234} (2019), no.~2, 575--593.

\bibitem[CQ]{CQ}
G.Q.~Chen, and Z.~Qian, 
\emph{A study of the {N}avier-{S}tokes equations with the kinematic
	and {N}avier boundary conditions},
Indiana Univ. Math. J.,~\textbf{59} (2010), no.~2, 721--760.

\bibitem[CW]{CW}
D.~Chae, and J.~Wolf,
\emph{The {E}uler equations in a critical case of the generalized
{C}ampanato space},
Ann. Inst. H. Poincar\'{e} C Anal. Non Lin\'{e}aire,~\textbf{38}, no.~2, (2021), 201--241.

\bibitem[DHP]{DHP} 
R.~Denk, M.~Hieber, and J.Pr\"uss~,   
\emph{Optimal {$L^p$}-{$L^q$}-estimates for parabolic boundary value
	problems with inhomogeneous data}, 
Math. Z.~\textbf{257} (2007), no.~1, 193--224. 

\bibitem[DN]{DN}
T.D.~Drivas, and H.Q.~Nguyen, 
\emph{Remarks on the emergence of weak {E}uler solutions in the
	vanishing viscosity limit},
J. Nonlinear Sci.,~\textbf{29} (2019), no.~2, 709--721.

\bibitem[FTZ]{FTZ} 
M.~Fei, T.~Tao, and Z.~Zhang,
\emph{On the zero-viscosity limit of the {N}avier-{S}tokes equations in {$\Bbb{R}_+^3$} without analyticity},
J.~Math. Pures Appl. (9) \textbf{112} (2018), 170--229.

\bibitem[Gu]{Gu} 
O.~Gu\`es,  
\emph{Probl\`eme mixte hyperbolique quasi-lin\'{e}aire
	caract\'{e}ristique}, 
Comm. Partial Differential Equations~\textbf{15} (1990), no.~5, 595--645. 

\bibitem[GK]{GK} 
G.M.~Gie, and J.P.~Kelliher,  
\emph{Boundary layer analysis of the {N}avier-{S}tokes equations
	with generalized {N}avier boundary conditions}, 
J. Differential Equations~\textbf{253} (2012), no.~6, 1862--1892. 

\bibitem[GKLMN]{GKLMN} 
G.-M.~Gie, J.P. Kelliher, M.C. Lopes~Filho, A.L. Mazzucato, and H.J. Nussenzveig~Lopes,
\emph{The vanishing viscosity limit for some symmetric flows},
Ann. Inst. H. Poincar\'{e} C Anal. Non Lin\'{e}aire~\textbf{36} (2019), no.~5, 1237--1280.

\bibitem[GL]{GL}
Z.~Guo, and J.~Li,
\emph{Remarks on the well-posedness of the {E}uler equations in the
	{T}riebel-{L}izorkin spaces},
J. Fourier Anal. Appl.,~\textbf{27}, no.~2, (2021), Paper No. 29, 24.

\bibitem[GLY]{GLY}
Z.~Guo, J.~Li, and Z.Yin,
\emph{Local well-posedness of the incompressible {E}uler equations
in {$B^1_{\infty,1}$} and the inviscid limit of the
{N}avier-{S}tokes equations},
J. Funct. Anal.,~\textbf{276}, no.~9, (2019), 2821--2830.

\bibitem[IP]{IP}
D.~Iftimie, and G.~Planas, 
\emph{Inviscid limits for the {N}avier-{S}tokes equations with
	{N}avier friction boundary conditions},
Nonlinearity,~\textbf{19} (2006), no.~4, 899--918.

\bibitem[IS]{IS}
D.~Iftimie, and F.~Sueur,
\emph{Viscous boundary layers for the {N}avier-{S}tokes equations
	with the {N}avier slip conditions},
Arch.\ Ration.\ Mech.\ Anal.~\textbf{199} (2011), no.~1, 145--175.

\bibitem[K1]{K1}
T.~Kato,
\emph{Nonstationary flows of viscous and ideal fluids in {${\bf
			R}\sp{3}$}},
J. Functional Analysis,~\textbf{9}, (1972), 296--305.

\bibitem[K2]{K2}
T.~Kato,
\emph{Quasi-linear equations of evolution, with applications to
partial differential equations},
Spectral theory and differential equations ({P}roc. {S}ympos.,
{D}undee, 1974; dedicated to {K}onrad {J}\"{o}rgens),
Lecture Notes in Math.,~\textbf{Vol. 448}, Springer, Berlin-New York
(1975), 25--70.

\bibitem[K3]{K3} 
T.~Kato,
\emph{Remarks on zero viscosity limit for nonstationary {N}avier-{S}tokes flows with boundary},
Seminar on nonlinear partial differential equations ({B}erkeley, {C}alif., 1983),
Math. Sci. Res. Inst. Publ., vol.~2, Springer, New York, 1984, pp.~85--98.

\bibitem[Ke]{Ke} 
J.P.~Kelliher,
\emph{Vanishing viscosity and the accumulation of vorticity on the boundary},
Commun. Math. Sci. \textbf{6} (2008), no.~4, 869--880.

\bibitem[KL]{KL}
T.~Kato, and C.Y.~Lai, 
\emph{Nonlinear evolution equations and the {E}uler flow},
J. Funct. Anal.,~\textbf{56}, (1984), 15--28.

\bibitem[KM]{KM} 
N.~Kajiwara, and A.Matsui~,   
\emph{Maximal regularity for the heat equation with various boundary conditions in an infinite layer}, 
SUT journal of mathematics~\textbf{59} (2023), no.~2, 73--90. 

\bibitem[KP]{KP}
T.~Kato, and G.~Ponce,
\emph{Commutator estimates and the {E}uler and {N}avier-{S}tokes
	equations},
Comm. Pure Appl. Math.,~\textbf{41}, no.~7, (1988), 891--907.

\bibitem[KVW]{KVW} 
I.~Kukavica, V.~Vicol, and F.~Wang,
\emph{The inviscid limit for the {N}avier-{S}tokes equations with data analytic only near the boundary},
Arch. Ration. Mech. Anal.,~\textbf{237} (2020), no.~2, 779--827.

\bibitem[L]{L}
L.~Lichtenstein,
\emph{\"{U}ber einige {E}xistenzprobleme der {H}ydrodynamik
	homogener, unzusammendr\"{u}ckbarer, reibungsloser
	{F}l\"{u}ssigkeiten und die {H}elmholtzschen
	{W}irbels\"{a}tze},
Math. Z.,~\textbf{23} (1925), no.~1, 89--154.

\bibitem[LNP]{LNP} 
M.C. Lopes~Filho, H.J. Nussenzveig~Lopes, and G.~Planas,
\emph{On the inviscid limit for two-dimensional incompressible flow
	with {N}avier friction condition},
SIAM J. Math. Anal.,~\textbf{36} (2005), no.~4, 1130--1141.

\bibitem[LSU]{LSU} 
O.A.~Lady\v zenskaja, V.A.~Solonnikov, and N.N.~Ural\textquotesingle ceva,  
\emph{Linear and quasilinear equations of parabolic type}, 
Translations of Mathematical Monographs~\textbf{Vol. 23},
American Mathematical Society, Providence, RI (1968). 

\bibitem[M]{M} 
Y.~Maekawa,
\emph{On the inviscid limit problem of the vorticity equations for viscous incompressible flows in the half-plane},
Comm. Pure Appl. Math. \textbf{67} (2014), no.~7, 1045--1128.

\bibitem[MM]{MM} 
Y.~Maekawa and A.~Mazzucato,
\emph{The inviscid limit and boundary layers for {N}avier-{S}tokes flows},
Handbook of mathematical analysis in mechanics of viscous fluids, Springer, Cham, 2018, pp.~781--828.

\bibitem[MR1]{MR1} 
N.~Masmoudi, and F.~Rousset,  
\emph{Uniform Regularity for the Navier–Stokes Equation with Navier Boundary Condition}, 
Arch.\ Ration.\ Mech.\ Anal.~\textbf{203} (2012), no.~2, 529--575. 

\bibitem[MR2]{MR2} 
N.~Masmoudi, and F.~Rousset,  
\emph{Uniform regularity and vanishing viscosity limit for the free
	surface {N}avier-{S}tokes equations}, 
Arch. Ration. Mech. Anal.~\textbf{223} (2017), no.~1, 301--417. 


\bibitem[NP]{NP}
J.~Neustupa, and P.~Penel, 
\emph{Approximation of a solution to the {E}uler equation by
	solutions of the {N}avier-{S}tokes equation},
J. Math. Fluid Mech.,~\textbf{15} (2013), no.~1, 179--196.

\bibitem[PP]{PP}
H.C.~Pak, and Y.J.~Park,
\emph{Existence of solution for the {E}uler equations in a critical
	{B}esov space {$ B^1_{\infty,1}(\Bbb R^n)$}},
Comm. Partial Differential Equations,~\textbf{29}, no.~7-8, (2004), 1149--1166.

\bibitem[SC1]{SC1} 
M.~Sammartino and R.E. Caflisch,
\emph{Zero viscosity limit for analytic solutions, of the {N}avier-{S}tokes equation on a half-space. {I}. {E}xistence for {E}uler and {P}randtl equations},
Comm. Math. Phys.~\textbf{192} (1998), no.~2, 433--461.

\bibitem[SC2]{SC2} 
M.~Sammartino and R.E.~Caflisch,
\emph{Zero viscosity limit for analytic solutions of the {N}avier-{S}tokes equation on a half-space. {II}. {C}onstruction of the {N}avier-{S}tokes solution},
Comm. Math. Phys.~\textbf{192} (1998), no.~2, 463--491.

\bibitem[Te]{Te}
R.~Temam,
\emph{On the {E}uler equations of incompressible perfect fluids},
J. Functional Analysis,~\textbf{20}, no.~1, (1975), 32--43.


\bibitem[TW]{TW} 
R.~Temam and X.~Wang,
\emph{On the behavior of the solutions of the {N}avier-{S}tokes equations at vanishing viscosity},
vol.~25, 1997, Dedicated to Ennio De Giorgi, pp.~807--828.

\bibitem[WXZ]{WXZ}
L.~Wang, Z.~Xin, and A.~Zang,
\emph{Vanishing viscous limits for 3{D} {N}avier-{S}tokes equations
with a {N}avier-slip boundary condition},
J. Math. Fluid Mech.,~\textbf{14} (2012), no.~4, 791--825.

\bibitem[X]{X}
Z.~Xin, 
\emph{Vanishing viscosity limits for the 3{D} {N}avier-{S}tokes
equations with a slip boundary condition},
Proc. Amer. Math. Soc.,~\textbf{145} (2017), no.~4, 1615--1628.

\bibitem[XX1]{XX1}
Y.~Xiao, and Z.~Xin, 
\emph{A new boundary condition for the three-dimensional {N}avier-{S}tokes equation and the vanishing viscosity limit},
J. Math. Phys.,~\textbf{53} (2012), no.~11, 115617, 15.

\bibitem[XX2]{XX2}
Y.~Xiao, and Z.~Xin,
\emph{On the inviscid limit of the 3{D} {N}avier-{S}tokes equations with generalized {N}avier-slip boundary conditions},
Commun. Math. Stat.,~\textbf{1} (2013), no.~3, 259--279.

\end{thebibliography}
\end{document}